\documentclass[12pt]{amsart}
\usepackage{latexsym,enumerate}
\usepackage{adjustbox}

\usepackage[utf8]{inputenc}
\usepackage[T1]{fontenc}
\usepackage{fullpage} 
\usepackage[nocolor]{sseq}
\usepackage[textwidth=1.2in, textsize=small]{todonotes} 
\usepackage{csquotes}
\usepackage{comment}
\usepackage{hyperref}
\usepackage{url}
\usepackage{float, graphicx, pinlabel}
\usepackage{tikz, tikz-cd} 
\usetikzlibrary{arrows}
\usetikzlibrary{matrix}
\usetikzlibrary{shapes}
\usetikzlibrary{calc}
\usetikzlibrary{shapes.geometric} 
\usetikzlibrary{angles,intersections,quotes}
\usetikzlibrary{decorations.markings}
\usetikzlibrary{graphs}
\usetikzlibrary{graphs.standard}
\usepackage{caption,subcaption}
\usepackage{thmtools} 
\usepackage{thm-restate}
\usepackage{mathabx}
\usepackage{amsmath,amscd,amssymb}

\theoremstyle{definition} 
\newtheorem{Theorem}{Theorem}[section]
\newtheorem{Question}{Question}
\newtheorem{Corollary}[Theorem]{Corollary}
\newtheorem{Lemma}[Theorem]{Lemma}

\newtheorem{Definition}[Theorem]{Definition}
\newtheorem*{remark}{Remark}

\theoremstyle{plain}

\newcommand{\gG}{\Gamma}
\newcommand{\BZ}{\mathbb{Z}}

\newcommand{\tConf}[1]{\mathrm{Conf}_{3}(\Theta_{#1})}

\DeclareMathOperator{\Lk}{Lk}
\DeclareMathOperator{\cConf}{Conf} 
\definecolor{darkgreen}{RGB}{0,128,0}

\title{Manifold models for hyperbolic graph braid groups on three strands}
\author{Saumya Jain, Huong Vo}

\begin{document}

\begin{abstract}
    Given a finite graph $\gG$, the associated \textit{graph braid group} $B_n(\gG)$ is the fundamental group of the unordered $n$-point configuration space of $\gG$. Genevois classified which graph braid groups are Gromov hyperbolic and asked the question: When do these groups arise as $3$-manifold groups? In this paper, we give a partial answer for $B_3(\Theta_m)$, where $\Theta_m$ is the \textit{generalized $\Theta$-graph}, a suspension of $m$-points. We show that $B_3(\Theta_5)$ is a $3$-manifold group while $B_3(\Theta_m)$ is not even quasi-isometric to a $3$-manifold group for $m \geq 7$.
   
\end{abstract}

\maketitle


\section{Introduction}

\noindent Configuration spaces of points on graphs were popularized by Ghrist~\cite{MR1873106} and Abrams \cite{MR2701024} as models for non-colliding motion of robots on a graph. 
Assuming that $n$ indistinguishable robots move on a finite graph $\gG$, their continuous non-colliding motion can be described by the topological configuration space $\cConf_n^{\text{top}}(\gG)$, defined as
  \[  \cConf_n^{\text{top}}(\gG) \coloneq (\gG^n - \Delta)/S_n, \]
  where $\Delta = \{(x_1, x_2, \dots x_n) \in \gG^n \text{ } |\text{ } x_i = x_j \text{ for some } i \ne j\}$ is the pairwise diagonal and $S_n$ is the symmetric group on $n$ letters.

These quotient spaces are aspherical, that is, $\pi_k(\cConf_n^{\text{top}}(\gG)) = 0$ for all $k>1$, and hence their homotopy type is determined by their fundamental group~\cite[Corollary~3.6]{MR2701024}. The groups $\pi_1 (\cConf_n^{\text{top}}(\gG))$ are called \emph{graph braid groups}, denoted by $B_n(\gG)$.

The space $\cConf^{\text{top}}_n(\gG)$ deformation retracts to a CW-complex which is easier to work with~\cite{MR1853416}\cite{MR2701024}. Roughly speaking, this is done by thinking of robots moving \emph{discretely} from one vertex to another along an edge. This is formalized by defining \emph{combinatorial configuration spaces}, denoted by $\cConf_n(\gG)$, as
\[    \cConf_n(\gG) \coloneq (\gG^n - \Delta^\Box)/S_n, \]  $\text{where } \Delta^\Box =\{ (e_1, e_2, \ldots, e_n) \mid e_i \text{ are edges in $\gG$, } e_i \cap e_j \neq \emptyset \text{ for some } i \neq j \}$ is the \enquote{thickened} pairwise diagonal. For example, for the tripod graph $Y$, $\cConf_2(Y)$ is the boundary of a hexagon (see Figure~\ref{fig:conf_2 of a tripod}).

The space $\cConf_n(\gG)$ naturally admits a cube complex structure described as follows. 
Each vertex in $\cConf_n(\gG)$ describes a configuration of $n$ indistinguishable robots on $n$ distinct vertices. 
Two vertices in $\cConf_n(\gG)$ are joined by an edge if we can go from one associated configuration to another by moving one of the robots across an edge without collisions. 
A $k$-dimensional cube corresponds to when $k$ such movements can be performed simultaneously. 
Equipped with the natural piecewise Euclidean metric, $\cConf_n(\gG)$ is locally CAT($0$) by Gromov's criterion~\cite[Proposition~2.3.1]{MR1853416} since the links of all vertices in the complex are flag complexes. In particular, the cube complexes $\cConf_n(\Gamma)$ are the classifying spaces of graph braid groups. 

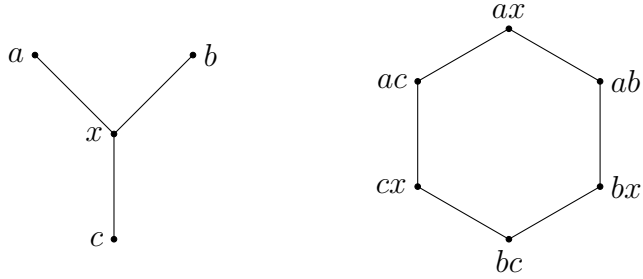
\begin{figure}[!t]
\adjustbox{max height=0.25\textheight}{%
    \begin{tikzpicture}[scale=0.7]
        \draw[black] (0,0)--(0,2)--(-1.5,3.5); 
        \draw[black] (0,2)--(1.5,3.5);
        \filldraw[black] (0,0) circle (1.5pt) node[left]{$c$};
        \filldraw[black] (0,2) circle (1.5pt) node[left]{$x$};
        \filldraw[black] (-1.5,3.5) circle (1.5pt) node[left] {$a$};
        \filldraw[black] (1.5,3.5) circle (1.5pt) node[right]{$b$};
        
        \draw[black] (7.5,0)--(5.768,1)--(5.768,3)--(7.5,4)--(9.232,3)--(9.232,1)--(7.5,0);
        \filldraw[black] (7.5,0) circle (1.5pt) node[below]{$bc$};
        \filldraw[black] (5.768,1) circle (1.5pt) node[left]{$cx$};
        \filldraw[black] (5.768,3) circle (1.5pt) node[left]{$ac$};
        \filldraw[black] (7.5,4) circle (1.5pt) node[above]{$ax$};
        \filldraw[black] (9.232,3) circle (1.5pt) node[right]{$ab$};
        \filldraw[black] (9.232,1) circle (1.5pt) node[right]{$bx$};
    \end{tikzpicture}
}
    \caption{ The tripod $Y$ (a cone on 3 points) and $\cConf_{2}(Y)$.}
    \label{fig:conf_2 of a tripod}
\end{figure}

The topological configuration space $\cConf_n ^{\text{top}}(\gG)$ depends only on the topology of the space $\gG$; however, its combinatorial counterpart $\cConf_n(\gG)$ depends on the CW structure of $\gG$ as well. 
Prue and Scrimshaw show that after sufficient subdivision (called an \textit{admissible} subdivision) of $\gG$, the combinatorial version is homotopy equivalent to the topological one~\cite[Theorem~3.2]{MR3276733}. 
Thus, we can limit our attention to $\cConf_n(\gG)$ when studying graph braid groups. 

Genevois has recently initiated a study of graph braid groups from the perspective of geometric group theory. 
In particular, he classified which graph braid groups on three or more strands are Gromov hyperbolic~\cite[Theorem~4.1]{MR4227231}. 
If $n \ge 4$, then hyperbolicity of $B_n(\gG)$ is exceedingly rare as Genevois showed that $\gG$ must be a so-called \emph{rose graph}, which implies that $B_n(\gG)$ is free~\cite[Lemma~4.6]{MR4227231}. 

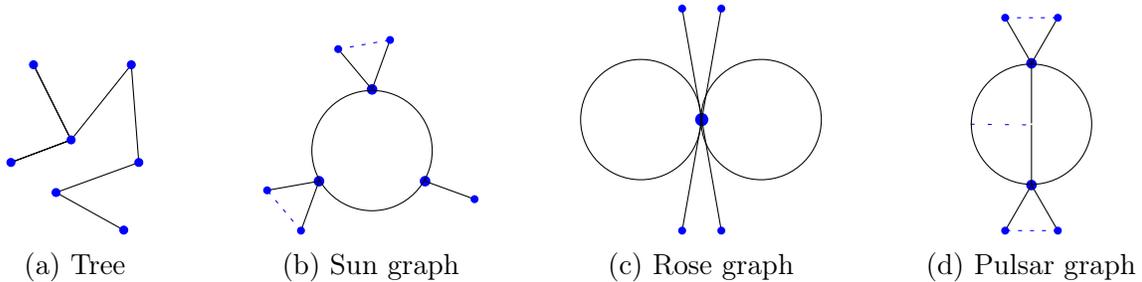
\begin{figure}[!t]
	\centering
	\begin{subfigure}{0.23\textwidth}
		\centering
        \begin{tikzpicture}
            \draw (0, 0) -- (-0.5, 1) -- (0, 0) -- (-0.8, -0.3) -- (0, 0) -- (0.8, 1) -- (0.9, -0.3) -- (-0.2, -0.7) -- (0.7, -1.2);
			\filldraw[blue] (0, 0) circle (1.5pt);
			\filldraw[blue] (-0.5, 1) circle (1.5pt);
			\filldraw[blue] (-0.8, -0.3) circle (1.5pt);
			\filldraw[blue] (0.8, 1) circle (1.5pt);
			\filldraw[blue] (0.9, -0.3) circle (1.5pt);
			\filldraw[blue] (-0.2, -0.7) circle (1.5pt);
			\filldraw[blue] (0.7, -1.2) circle (1.5pt);
        \end{tikzpicture}
	\caption{Tree}
	\label{fig:tree}
	\end{subfigure}
	\centering
	\begin{subfigure}{0.23\textwidth}
		\centering
		\begin{tikzpicture}
            \node[draw=none,minimum size=1.6cm,regular polygon,regular polygon sides=3] (n) {};
            
            \foreach \x in {1,2,3}
              {
                \fill[blue] (n.corner \x) circle[radius=2pt];
                
                \ifnum\x<3
                    \draw (n.corner \x) -- ++(\x*360/3+10:0.7)
                        coordinate(endA);
                    \fill[blue] (endA) circle[radius=1.5pt];
                    
                    \draw (n.corner \x) -- ++(\x*360/3-50:0.7)
                        coordinate(endB);
                    \fill[blue] (endB) circle[radius=1.5pt];
                    
                    \draw[blue, dotted, dash pattern=on 1.5pt off 3pt] (endA) -- (endB);
                \else 
                    \draw (n.corner \x) -- ++(\x*360/3-20:0.7)
                        coordinate(endC);
                    \fill[blue] (endC) circle[radius=1.5pt];
                \fi
              }
            \draw (0,0) circle (0.8cm);
		\end{tikzpicture}
		\caption{Sun graph}
		\label{fig:sun-graph}
	\end{subfigure}
	\hfill
	\begin{subfigure}{0.23\textwidth}
		\centering
		\begin{tikzpicture}
            \node[draw=none,minimum size=1.6cm,regular polygon,regular polygon sides=4,shape border rotate=45] (A) at (0,0) {};
            \node[draw=none,minimum size=1.6cm,regular polygon,regular polygon sides=4,shape border rotate=45] (B) at (1.6,0) {};
            
            \fill[blue] (A.east) circle[radius=2.5pt];

            \draw (A.east) -- ++(100:1.5) 
                coordinate(endA);
            \draw (A.east) -- ++(80:1.5) 
                coordinate(endB);
            \draw (A.east) -- ++(-80:1.5) 
                coordinate(endC);
            \draw (A.east) -- ++(-100:1.5) 
                coordinate(endD);

            \fill[blue] (endA) circle[radius=1.5pt];
            \fill[blue] (endB) circle[radius=1.5pt];
            \fill[blue] (endC) circle[radius=1.5pt];
            \fill[blue] (endD) circle[radius=1.5pt];
            
            \draw (0,0) circle (0.8cm);
            \draw (1.6,0) circle (0.8cm);
		\end{tikzpicture}
		\caption{Rose graph}
		\label{fig:rose-graph}
	\end{subfigure}
	\hfill
	\begin{subfigure}{0.23\textwidth}
		\centering
		\begin{tikzpicture}
            \node[draw=none,minimum size=1.6cm,regular polygon,regular polygon sides=4,shape border rotate=45] (n) {};
            
            \foreach \x in {1,3}
              {
                \fill[blue] (n.corner \x) circle[radius=2pt];
                
                \draw (n.corner \x) -- ++(\x*360/4+30:0.7)
                    coordinate(endA);
                \fill[blue] (endA) circle[radius=1.5pt];
                
                \draw (n.corner \x) -- ++(\x*360/4-30:0.7)
                    coordinate(endB);
                \fill[blue] (endB) circle[radius=1.5pt];

                \draw (n.corner \x) -- ++(\x*360*3/4:0.8)
                    coordinate(endC);
                \draw[blue, dotted, dash pattern=on 1.5pt off 3pt] (endA) -- (endB);
              }
            \draw[blue, dotted, dash pattern=on 1.5pt off 5pt] (n.corner 2) -- (endC);
            \draw (0,0) circle (0.8cm);
		\end{tikzpicture}
		\caption{Pulsar graph}
		\label{fig:pulsar-graph}
	\end{subfigure}
	\caption{The families of graphs $\gG$ for which $B_3(\gG)$ is hyperbolic.}
	\label{fig:family-hyperbolic-graphs}
\end{figure}

In the $3$-strand case, he showed that there are four classes of graphs -- trees,  sun graphs,  rose graphs, and  pulsar graphs, as shown in Figure~\ref{fig:family-hyperbolic-graphs} -- which give hyperbolic $B_3(\gG)$. Furthermore, he asked:
\begin{Question}
    When are the hyperbolic graph braid groups on three strands the fundamental groups of $3$-manifolds?
\end{Question} 

\noindent Again, in most cases the graph braid group turns out to be free. Specifically, when $\gG$ is a tree graph, a sun graph, or a rose graph, $B_3(\gG)$ is free by the results of Farley and Sabalka~\cite[Theorem~2.5 and Theorem~4.3]{MR2171804}, Appiah et al.~\cite[Theorem~1.3]{MR4901846}, and Genevois~\cite[Lemma 4.6]{MR4227231}, respectively. 
For the pulsar graph, $B_3(\gG) \cong B_3(\Theta_m) \ast \mathbb{F}$, where $\mathbb{F}$ is a free group of finite rank and $\Theta_m$ is a suspension on $m$ points~\cite[Theorem~1.4]{MR4901846}.
We will therefore be interested in investigating when the $3$-strand graph braid group $B_3(\Theta_m)$ is a $3$-manifold group.

The first obstruction to this is the Euler characteristic of $\tConf{m}$. It is an easy calculation to show that $\chi(\tConf{m})$ is $\frac{m(m-2)(m-7)}{6}$.
We observe that $\chi(\tConf{m}) > 0$ when $m > 7$. This implies that $B_3(\Theta_m)$, for any $m > 7$, is not the fundamental group of a $3$-manifold (see the remark after Corollary~\ref{cor:3m} for more details). For $m \le 4$, these are surface groups. In particular, $B_3(\Theta_4)$ is the fundamental group of the closed, connected, orientable genus $3$ surface $\Sigma_3$~\cite{MR3797072}~\cite{MR3887191}. 

In this paper, we provide a positive answer to Genevois' question for $m=5$ and a negative answer for $m = 7$. 

\begin{Theorem}\label{thm:main5}
    The configuration space $\tConf{5}$ thickens to an orientable $3$-manifold. Thus, $B_3(\Theta_5)$ is a $3$-manifold group.
\end{Theorem}

\noindent To prove this, we first subdivide $\tConf{5}$ to form a finite 2-dimensional complex with planar links. Not all such complexes can be thickened to $3$-manifolds. For example, the presentation complex of the Baumslag-Solitar subgroup $BS(1,2)$, whose vertex link is a suspension of three points, admits no such thickening as any $BS(m, n)$ for $m \ne n$ is not a $3$-manifold group~\cite[Proposition~2]{MR438320}. We use a theorem of Lasheras~\cite[Theorem~1.1]{MR1621973} which says a thickening to a $3$-manifold exists if and only if a certain cocycle 
is trivial. (This cocycle roughly measures the \enquote{twists} in the complex that cannot be \enquote{untwisted} and hence presents an obstruction to the thickening of the complex to a $3$-manifold.) 
We construct a family of embeddings into $\mathbb{R}^2$ so that the associated cochain is trivial, thus allowing for a $3$-thickening of the complex.

For $m=7$, we find that such a thickening is not possible. In fact:

\begin{Theorem}\label{thm:main7}
    $B_3(\Theta_7)$ is not quasi-isometric to a $3$-manifold group.
\end{Theorem}

\noindent To show this, we find a complete bipartite graph $K_{3,3}$ in the boundary of $B_3(\Theta_7)$. The $K_{3,3}$ graph is built by stitching together paths in the boundary coming from certain embedded surfaces in $\tConf{7}$. 
The presence of these subsurfaces is induced by the natural inclusions $\Theta_4 \hookrightarrow \Theta_7$.
In turn, the work of Bestvina--Kapovich--Kleiner~\cite{MR1933584} says that $B_3(\Theta_7)$ is not a $3$-manifold group. Since the Gromov boundary is a quasi-isometric invariant, this implies Theorem~\ref{thm:main7}. 

In contrast to $\tConf{5}$, the vertex links in $\tConf{6}$ are not all planar. So Lasheras' thickening criterion cannot be applied. Moreover, there are not enough $\Theta_4$ inclusions in $\Theta_6$ to confirm the presence of a $K_{3,3}$ in the boundary of $B_3(\Theta_6)$. Thus, we are left with the question:
\begin{Question}
    Is $\tConf{6}$ homotopy equivalent to a $3$-manifold? Is $B_3(\Theta_6)$ a $3$-manifold group, even up to quasi-isometry?
\end{Question}

\subsection*{Acknowledgments}
The authors would like to thank Kevin Schreve for introducing us to graph braid groups, and for helpful advising, numerous stimulating discussions and support throughout this project. The authors are also thankful to Pallavi Dani for her constructive comments on the drafts of this paper.
For the duration of this work, the first author was supported by the U.S. National Science Foundation under Award Number DMS-2407104 and the Department of Mathematics at Louisiana State University. Both authors were also supported by the U.S. National Science Foundation under Grant DMS-2203325.


\section{\texorpdfstring{$B_3(\Theta_5)$ is a $3$-manifold group}{B3(Theta5) is a 3-manifold group}}

\noindent In this section, we show that the combinatorial configuration space $\tConf{5}$ thickens to an orientable $3$-manifold. We will use the minimal admissible subdivision of the $\Theta_5$ graph, as shown in Figure~\ref{fig:theta5-graph}. The subdivided graph $\Theta_5$ does not contain three pairwise disjoint edges. Thus, the complex $\mathrm{Conf}_3(\Theta_5)$ is a connected, locally-finite square complex.
We begin by recalling Lasheras' thickening obstruction~\cite{MR1621973}, which is stated for simplicial complexes and then describe a modification suited to square complexes.  

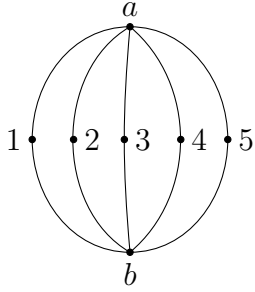
\begin{figure}[!t]
    \centering
    \begin{tikzpicture}
    \tikzset{every picture/.style={line width=1pt}} 
        \filldraw[black] (0, 1.5) circle (1.2pt) node[anchor=south]{$a$};
        \filldraw[black] (0, -1.5) circle (1.2pt) node[anchor=north]{$b$};
        \filldraw[black] (-1.3, 0) circle (1.2pt) node[anchor=east]{$1$};
        \filldraw[black] (-0.75, 0) circle (1.2pt) node[anchor=west]{$2$};
        \filldraw[black] (-0.075, 0) circle (1.2pt) node[anchor=west]{$3$};
        \filldraw[black] (0.675, 0) circle (1.2pt) node[anchor=west]{$4$};
        \filldraw[black] (1.3, 0) circle (1.2pt) node[anchor=west]{$5$};
        
        \draw (0, 1.5) to[out=180, in=90] (-1.3, 0) to[out=-90, in=180] (0, -1.5);
        \draw (0, 1.5) to[out=210, in=150] (0, -1.5);
        \draw (0, 1.5) to[out=-95, in=95] (0, -1.5);
        \draw (0, 1.5) to[out=-40, in=40] (0, -1.5);
        \draw (0, 1.5) to[out=0, in=90] (1.3, 0) to[out=-90, in=0] (0, -1.5);
    \end{tikzpicture}
	\caption{The $\Theta_5$ graph.}
	\label{fig:theta5-graph}
\end{figure}

\begin{Definition}[{\cite[Definition~3.1]{MR1621973}}]\label{def:lasheras-embedding}
Let $X$ be a $2$-dimensional locally finite simplicial complex in which the link $\Lk{(v, X)}$ of every vertex $v$ of $X$ is planar. Consider an embedding $\phi_v: \Lk{(v, X)} \longrightarrow \mathbb{R}^2$. Following the clockwise orientation around $\phi_{v}{(w)}$ in $\mathbb{R}^2$, where $w \in \Lk{(v, X)}$, we obtain a tuple that is associated with the vertices adjacent to $w$ in $\Lk{(v, X)}$. We call this tuple the cyclic ordering on $\Lk{(\overline{vw}, X)}$ determined by $\phi_v$ and denote it by $\theta_{\phi_v}(w)$.   

We note that if the edge $e = \overline{vw}$ is contained in at most two faces, then there is only one cyclic ordering on $\Lk{(e, X)}$. If $e$ is contained in exactly three faces, there are two such possible cyclic orderings.

\end{Definition}
\begin{Definition}[{\cite{MR1621973}}]\label{def:lasheras-cochain}
Let $X$ be as in Definition~\ref{def:lasheras-embedding} and $X^{(1)}$ denote its $1$-skeleton. Let $X'$ be the subgraph of $X^{(1)}$ consisting of edges $e = \overline{vw}$ of $X$ which are contained in at least three faces. Consider the cochain complex of $X'$ over $\BZ_2$,
\[
0 \longrightarrow C^0(X'; \BZ_2) \longrightarrow  C^1(X'; \BZ_2) \longrightarrow 0.
\]

\noindent Given a family of embeddings $\Phi = \{\phi_v: \Lk{(v,X)} \longrightarrow \mathbb{R}^2\}$, we associate a cocycle   
\[ 
\omega_\Phi = \sum_{e \in X'} \omega_\Phi (e) \cdot e \in C^1(X'; \BZ_2),
\]
where
\[ 
\omega_{\Phi}(e) = 
\begin{cases}
0  \text{ if } \theta_{\phi_v}(w) \text{ and } \theta_{\phi_w}(v) \text{ are opposite cyclic orders on } \Lk{(e,X)}, \\
1 \text{ otherwise.}
\end{cases}
\]
\end{Definition}

\noindent In~\cite[Theorem~1.1]{MR1621973}, Lasheras proves that such a simplicial complex thickens to an orientable $3$-manifold if and only if $\omega_{\Phi}$ is trivial for some family of embeddings $\Phi$. We now adapt this result to square complexes. A square complex can be subdivided to form a simplicial complex by adding barycenters of squares. This does not change the homeomorphism type of links of the original vertices and the new barycenters have links isomorphic to $S^1$. Moreover, the new edges are contained in exactly two faces, and therefore, do not contribute to the cochain $\omega_{\Phi}$ for any family of embeddings $\Phi$. Therefore, Theorem~1.1 of {\cite{MR1621973} implies:

\begin{Theorem}\label{thm:lasheras-cube}
A $2$-dimensional connected locally finite cube complex $X$ thickens to an orientable $3$-manifold if and only if 
	\begin{enumerate}[(i)]
		\item $\Lk{(v, X)}$ is planar for every vertex $v$ of $X$. \label{thm:lasheras-cube-i}
		\item There exists a family of embeddings $\Phi = \{\phi_v: \Lk{(v, X)} \longrightarrow \mathbb{R}^2\}$ so that the associated cochain $\omega_{\Phi}$ is trivial.  \label{thm:lasheras-cube-ii}
	\end{enumerate}
\end{Theorem}

To apply Theorem~\ref{thm:lasheras-cube}, we need to understand the links of vertices in $\tConf{m}$. 

\begin{remark}
	\label{rem:lk} Let $v_1v_2v_3$ be a vertex in $\tConf{m}$, where each $v_i$ is a vertex of $\Theta_m$.  A vertex in $\Lk(v_1v_2v_3, \tConf{m})$ corresponds to an edge of $\Theta_m$ containing one of the vertices $v_i$ and disjoint from the other two. An edge in $\Lk(v_1v_2v_3, \tConf{m})$ corresponds to two disjoint edges in $\Theta_m$ each containing one of the vertices $v_i$. 
\end{remark}

\noindent Since there are no $3$-cubes in $\tConf{m}$, there are no $2$-simplices in the links of vertices. We now use Theorem~\ref{thm:lasheras-cube} to show that the square complex $\tConf{5}$ thickens to an orientable $3$-manifold.

\begin{proof}[Proof of Theorem~\ref{thm:main5}]

Let $a$ and $b$ be the suspension vertices and $1\leq i,j,k \leq 5$ be the suspended vertices in $\Theta_5$, as shown in Figure~\ref{fig:theta5-graph}. Then the four types of vertices in the square complex $\tConf{5}$ are $abi$, $aij$, $bij$, and $ijk$. The links of all vertices in $\tConf{5}$ are planar (see Figure~\ref{fig:theta5-links}). Therefore, the first assumption of Theorem~\ref{thm:lasheras-cube} is satisfied.

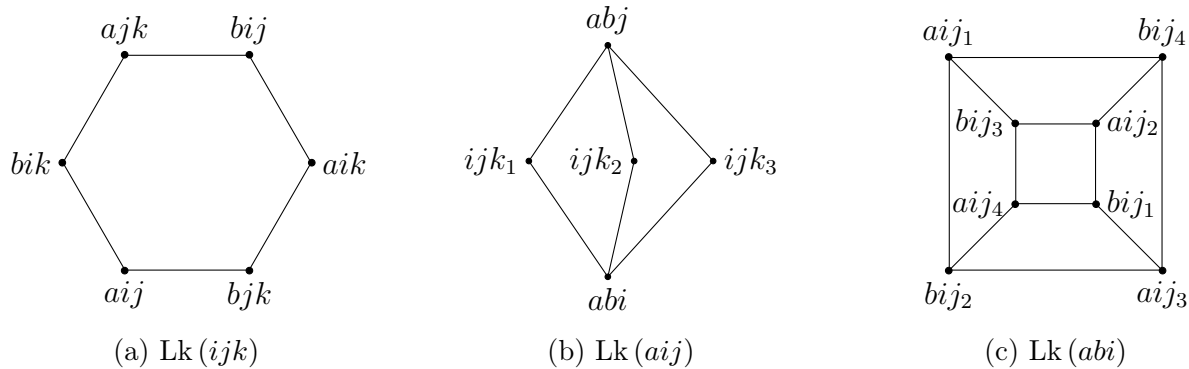
\begin{figure}
	\centering
	\begin{subfigure}{0.3\textwidth}
		\centering
		\begin{tikzpicture}[scale=0.5]
			\node[regular polygon, regular polygon sides=6, draw, minimum size=3.3cm] (hexagon) {};
			\filldraw (hexagon.corner 1) circle (2.5pt) node[anchor=south] {$bij$};
			\filldraw (hexagon.corner 2) circle (2.5pt) node[anchor=south] {$ajk$};
			\filldraw (hexagon.corner 3) circle (2.5pt) node[anchor=east] {$bik$};
			\filldraw (hexagon.corner 4) circle (2.5pt) node[anchor=north] {$aij$};
			\filldraw (hexagon.corner 5) circle (2.5pt) node[anchor=north] {$bjk$};
			\filldraw (hexagon.corner 6) circle (2.5pt) node[anchor=west] {$aik$};
		\end{tikzpicture}
		\caption{$\Lk{(ijk)}$}
		\label{fig:theta5-links-ijk}
	\end{subfigure}
	\hfill
	\begin{subfigure}{0.3\textwidth}
		\centering
		\begin{tikzpicture}[scale=0.7]
			\draw (0, 2.2) -- (-1.5, 0) -- (0, -2.2) -- (0.5, 0) -- (0, 2.2) -- (2, 0) -- (0, -2.2);
			\filldraw[black] (0, 2.2) circle (1.5pt) node[anchor=south] {$abj$};
			\filldraw[black] (0, -2.2) circle (1.5pt) node[anchor=north] {$abi$};
			\filldraw[black] (-1.5, 0) circle (1.5pt) node[anchor=east] {$ijk_1$};
			\filldraw[black] (0.5, 0) circle (1.5pt) node[anchor=east] {$ijk_2$};
			\filldraw[black] (2, 0) circle (1.5pt) node[anchor=west] {$ijk_3$};
		\end{tikzpicture}
		\subcaption{$\Lk{(aij)}$}
		\label{fig:theta5-links-aij}
	\end{subfigure}
	\hfill
	\begin{subfigure}{0.3\textwidth}
		\centering
		\begin{tikzpicture}[scale=0.5]
			\node[regular polygon, regular polygon sides=4, draw, minimum size=4cm] (osquare) {};
			\node[regular polygon, regular polygon sides=4, draw, minimum size=1.5cm] (isquare) {};
			\draw (osquare.corner 1) -- (isquare.corner 1);
			\draw (osquare.corner 2) -- (isquare.corner 2);
			\draw (osquare.corner 3) -- (isquare.corner 3);
			\draw (osquare.corner 4) -- (isquare.corner 4);
			\filldraw (osquare.corner 1) circle (2.5pt) node[anchor=south] {$bij_4$};
			\filldraw (osquare.corner 2) circle (2.5pt) node[anchor=south] {$aij_1$};
			\filldraw (osquare.corner 3) circle (2.5pt) node[anchor=north] {$bij_2$};
			\filldraw (osquare.corner 4) circle (2.5pt) node[anchor=north] {$aij_3$};
			\filldraw (isquare.corner 1) circle (2.5pt) node[anchor=west] {$aij_2$};
			\filldraw (isquare.corner 2) circle (2.5pt) node[anchor=east] {$bij_3$};
			\filldraw (isquare.corner 3) circle (2.5pt) node[anchor=east] {$aij_4$};
			\filldraw (isquare.corner 4) circle (2.5pt) node[anchor=west] {$bij_1$};
		\end{tikzpicture}
		\caption{$\Lk{(abi)}$}
		\label{fig:theta5-links-abi}
	\end{subfigure}
	\caption{Links of $ijk$, $aij$, and $abi$ in $\tConf{5}$.}
	\label{fig:theta5-links}
\end{figure}

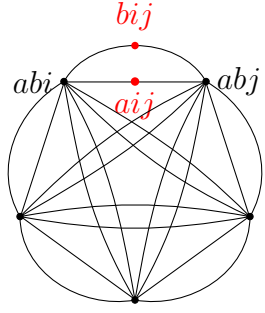
\begin{figure}[!t]
    \centering
    \begin{tikzpicture}[scale=0.8]
            \node[rotate=36, transform shape, regular polygon, regular polygon sides=5, draw, minimum size=4cm] (pentagon) {};
            
            \filldraw (pentagon.corner 1) circle (1.5pt) node[anchor=east] {$abi$};
            \filldraw (pentagon.corner 2) circle (1.5pt);
            \filldraw (pentagon.corner 3) circle (1.5pt);
            \filldraw (pentagon.corner 4) circle (1.5pt);
            \filldraw (pentagon.corner 5) circle (1.5pt) node[anchor=west] {$abj$};

            \draw[fill=red, draw=none]($(pentagon.corner 1)!0.5!(pentagon.corner 5)$) circle (2pt) node[red,anchor=north] {$aij$};
            \draw[
                postaction={
                    decorate,
                    decoration={
                        markings,
                        mark=at position 0.5 with {\node[fill=red, circle, inner sep=1pt, label={[red]above:{$bij$}}] {};}
                    }
                }
            ] (pentagon.corner 1) to[out=60, in=120] (pentagon.corner 5);

            \draw (pentagon.corner 1) to[out=210, in=120] (pentagon.corner 2);

            \draw (pentagon.corner 1) to[out=280, in=120] (pentagon.corner 3);
            \draw (pentagon.corner 1) to[out=300, in=100] (pentagon.corner 3);

            \draw (pentagon.corner 1) to[out=315, in=155] (pentagon.corner 4);
            \draw (pentagon.corner 1) to[out=330, in=135] (pentagon.corner 4);

            \draw (pentagon.corner 2) to[out=270, in=190] (pentagon.corner 3);

            \draw (pentagon.corner 2) to[out=350, in=190] (pentagon.corner 4);
            \draw (pentagon.corner 2) to[out=10, in=170] (pentagon.corner 4);

            \draw (pentagon.corner 2) to[out=45, in=205] (pentagon.corner 5);
            \draw (pentagon.corner 2) to[out=30, in=225] (pentagon.corner 5);

            \draw (pentagon.corner 3) to[out=350, in=270] (pentagon.corner 4);

            \draw (pentagon.corner 3) to[out=60, in=260] (pentagon.corner 5);
            \draw (pentagon.corner 3) to[out=80, in=240] (pentagon.corner 5);

            \draw (pentagon.corner 4) to[out=60, in=330] (pentagon.corner 5);
    \end{tikzpicture}
	\caption{The underlying graph $X'$ for $\tConf{5}$.
    }
	\label{fig:theta5-cochain-doubledK5}
\end{figure}

We need to show that the second condition of Theorem~\ref{thm:lasheras-cube} holds. That is, we need to find an explicit family of embeddings $\Phi = \{\phi_v\}_{v\in\tConf{5}}$ into $\mathbb{R}^2$ such that the associated cochain $\omega_\Phi$ is trivial. 

Observe from Figure~\ref{fig:theta5-links-ijk} that the degree of each $w \in \Lk{(ijk, \tConf{5})}$ is two. So the edges $\overline{(ijk)(aij)}$ or $\overline{(ijk)(bij)}$ in $\tConf{5}$ are contained in exactly two faces and therefore do not contribute to the cocycle associated to any embedding of links. 

Thus, we need to embed the links of vertices of the form $aij$, $bij$, and $abi$ so that the associated cochain is trivial. First, consider the subgraph $X'$ of the $1$-skeleton of $\tConf{5}$ as in Definition~\ref{def:lasheras-cochain}. Each vertex of the type $abi$ is joined to the vertices $aij$ and $bij$ for $i \neq j$, which in turn are joined to the vertex $abj$. This is true for all  $i \ne j \in \{1,2,3,4,5\}$. Thus, $X'$ is isomorphic to a barycentrically-subdivided double-edged $K_5$ (see Figure~\ref{fig:theta5-cochain-doubledK5}). If we find a family of embeddings $\Phi$ of the links of the vertices of $X'$ such that the associated $\omega_\Phi(e) = 0$ for each edge $e$ of $X'$, we will be done.

\begin{figure}[!t]
	\centering
    \begin{subfigure}[!t]{\textwidth}
    \centering
        \begin{tikzpicture}[scale=0.75]
            \coordinate (a12) at (0, 0);
            \coordinate (ab2) at (4, 0);
            \coordinate (ab1) at (-4, 0);
            \coordinate (a13) at (-5, 1.5);
            \coordinate (a14) at (-3, 0.75);
            \coordinate (a15) at (-4, -1.5);
            \coordinate (a23) at (3, 1.5);
            \coordinate (a24) at (5, 0.75);
            \coordinate (a25) at (4, -1.5);
            \coordinate (123) at (-1, 1.5);
            \coordinate (124) at (1, 0.75);
            \coordinate (125) at (0, -1.5);
			\draw (a12) -- (ab1) -- (a13) -- (123) -- (a12) -- (ab1) -- (a14) -- (124) -- (a12) -- (ab1) -- (a15) -- (125) -- (a12);
			\draw (a12) -- (ab2) -- (a23) -- (123) -- (a12) -- (ab2) -- (a24) -- (124) -- (a12) -- (ab2) -- (a25) -- (125) -- (a12);
            
			\filldraw[black] (a12) circle (1.5pt) node[anchor=north west] {$a12$};
			\filldraw[black] (ab2) circle (1.5pt) node[anchor=west] {$ab2$};
			\filldraw[black] (ab1) circle (1.5pt) node[anchor=east] {$ab1$};
			\filldraw[black] (a13) circle (1.5pt) node[anchor=south] {$a13$};
			\filldraw[black] (a14) circle (1.5pt) node[anchor=south] {$a14$};
			\filldraw[black] (a15) circle (1.5pt) node[anchor=north] {$a15$};
			\filldraw[black] (a23) circle (1.5pt) node[anchor=south] {$a23$};
			\filldraw[black] (a24) circle (1.5pt) node[anchor=south] {$a24$};
			\filldraw[black] (a25) circle (1.5pt) node[anchor=north] {$a25$};
			\filldraw[black] (123) circle (1.5pt) node[anchor=south] {$123$};
			\filldraw[black] (124) circle (1.5pt) node[anchor=south] {$124$};
			\filldraw[black] (125) circle (1.5pt) node[anchor=north] {$125$};

        \end{tikzpicture}
		\caption{Local neighborhood of $a12$ in $\tConf{5}$}
		\label{fig:}
    \end{subfigure}
    \vspace{0.2cm}
    
    \centering
	\begin{subfigure}[!t]{0.3\textwidth}
		\centering
        
		\begin{tikzpicture}[scale=0.5]
			\node[regular polygon, regular polygon sides=4, draw, minimum size=3.9cm] (osquare) {};
			\node[regular polygon, regular polygon sides=4, draw, minimum size=1.5cm] (isquare) {};
            
			\draw (osquare.corner 1) -- (isquare.corner 1);
			\draw (osquare.corner 2) -- (isquare.corner 2);
			\draw (osquare.corner 3) -- (isquare.corner 3);
			\draw (osquare.corner 4) -- (isquare.corner 4);
            
			\filldraw (osquare.corner 1) circle (2.5pt) node[anchor=south] {$b1\color{blue}5$};
			\filldraw (osquare.corner 2) circle (2.5pt) node[anchor=south] {$a12$};
			\filldraw (osquare.corner 3) circle (2.5pt) node[anchor=north] {$b1\color{blue}3$};
			\filldraw (osquare.corner 4) circle (2.5pt) node[anchor=north] {$a14$};
			\filldraw (isquare.corner 1) circle (2.5pt) node[anchor=west] {$a13$};
			\filldraw (isquare.corner 2) circle (2.5pt) node[anchor=east] {$b1\color{blue}4$};
			\filldraw (isquare.corner 3) circle (2.5pt) node[anchor=east] {$a15$};
			\filldraw (isquare.corner 4) circle (2.5pt) node[anchor=west] {$b12$};

            \coordinate (a12) at (osquare.corner 2);
            \coordinate (b15) at (osquare.corner 1);
            \coordinate (b14) at (isquare.corner 2);
            \coordinate (b13) at (osquare.corner 3);
			\draw [red, opacity=1] 
				pic [draw=blue, angle radius=0.6cm, <-] {angle = b14--a12--b13};
		\end{tikzpicture}
		\caption{$\Lk{(ab1)}$}
		\label{fig:theta5-links-ab1}
	\end{subfigure}
	\hfill
	\begin{subfigure}[!t]{0.3\textwidth}
		\centering
		\begin{tikzpicture}[scale=0.75, rotate=-90]
            \coordinate (ab2) at (0, 2.2);
            \coordinate (ab1) at (0, -2.2);
            \coordinate (123) at (-1.5, 0);
            \coordinate (124) at (0.5, 0);
            \coordinate (125) at (2, 0);
			\draw (ab2) -- (123) -- (ab1) -- (124) -- (ab2) -- (125) -- (ab1);
            
			\filldraw[black] (ab2) circle (1.5pt) node[anchor=west] {$ab2$};
			\filldraw[black] (ab1) circle (1.5pt) node[anchor=east] {$ab1$};
			\filldraw[black] (123) circle (1.5pt) node[anchor=south] {$12\color{teal}3$};
			\filldraw[black] (124) circle (1.5pt) node[anchor=south] {$12\color{teal}4$};
			\filldraw[black] (125) circle (1.5pt) node[anchor=north] {$12\color{teal}5$};
            
			\draw [red, opacity=1] 
				pic [draw=red, angle radius=0.7cm, <-] {angle = 125--ab1--123};
			\draw [blue, opacity=1] 
				pic [draw=blue, angle radius=0.7cm, <-] {angle = 124--ab2--123};
		\end{tikzpicture}
		\caption{$\Lk{(a12)}$}
		\label{fig:theta5-links-a12}
	\end{subfigure}
	\hfill
	\begin{subfigure}[!t]{0.3\textwidth}
		\centering
		\begin{tikzpicture}[scale=0.5]
			\node[regular polygon, regular polygon sides=4, draw, minimum size=3.9cm] (osquare) {};
			\node[regular polygon, regular polygon sides=4, draw, minimum size=1.5cm] (isquare) {};
            
			\draw (osquare.corner 1) -- (isquare.corner 1);
			\draw (osquare.corner 2) -- (isquare.corner 2);
			\draw (osquare.corner 3) -- (isquare.corner 3);
			\draw (osquare.corner 4) -- (isquare.corner 4);
            
			\filldraw (osquare.corner 1) circle (2.5pt) node[anchor=south] {$b2\color{red}4$};
			\filldraw (osquare.corner 2) circle (2.5pt) node[anchor=south] {$a23$};
			\filldraw (osquare.corner 3) circle (2.5pt) node[anchor=north] {$b2\color{red}5$};
			\filldraw (osquare.corner 4) circle (2.5pt) node[anchor=north] {$a12$};
			\filldraw (isquare.corner 1) circle (2.5pt) node[anchor=west] {$a25$};
			\filldraw (isquare.corner 2) circle (2.5pt) node[anchor=east] {$b12$};
			\filldraw (isquare.corner 3) circle (2.5pt) node[anchor=east] {$a24$};
			\filldraw (isquare.corner 4) circle (2.5pt) node[anchor=west] {$b2\color{red}3$};
            
            \coordinate (a12) at (osquare.corner 4);
            \coordinate (b24) at (osquare.corner 1);
            \coordinate (b23) at (isquare.corner 4);
            \coordinate (b25) at (osquare.corner 3);
			\draw [blue, opacity=1] 
				pic [draw=red, angle radius=0.6cm, <-] {angle = b25--a12--b23};
		\end{tikzpicture}
		\caption{$\Lk{(ab2)}$}
		\label{fig:theta5-links-ab2}
	\end{subfigure}
    \vspace{0.5cm}
    
	\begin{subfigure}[!t]{0.65\textwidth}
		\centering
		\begin{tikzpicture}[scale=0.75]
			\coordinate (ab1) at (0, 0);
			\coordinate (a12) at (2, 0);
			\coordinate (ab2) at (4, 0);
			
			\draw (ab1) edge node[red,anchor=south] {0} (a12);
            \draw (a12) edge node[red,anchor=south] {0} (ab2);
			
			\filldraw[black] (ab1) circle (2.5pt) node[anchor=east] {$ab1$};
			\filldraw[red] (a12) circle (2.5pt) node[anchor=north] {$a12$};
			\filldraw[black] (ab2) circle (2.5pt) node[anchor=west] {$ab2$};
		\end{tikzpicture}
        \hspace{0.7cm}\tikz[baseline=-1\baselineskip]\draw[ultra thick,<->] (0,0) -- ++ (0.7,0);\qquad
        \begin{tikzpicture}[scale=0.75]
			\coordinate (K5ab1) at (7, 0);
			\coordinate (K5ab2) at (10, 0);
			
			\draw (K5ab1) edge node[red,anchor=south] {0} (K5ab2);
			
			\filldraw[black] (K5ab1) circle (2.5pt) node[anchor=north] {$ab1$};
			\filldraw[black] (K5ab2) circle (2.5pt) node[anchor=north] {$ab2$};
		\end{tikzpicture}
	\label{fig:theta5-edgevalue-ab1-a12-ab2}
	\end{subfigure}
	\caption{The orderings $\theta_{\phi_{a12}}(ab1) = (345)$ and $\theta_{\phi_{a12}}(ab2) = (354)$ are opposite, $\theta_{\phi_{ab1}}(a12) = (354)$ and $\theta_{\phi_{a12}}(ab1) = (345)$ are opposite, and $\theta_{\phi_{ab2}}(a12) = (345)$ and $\theta_{\phi_{a12}}(ab2) = (354)$ are opposite. The path $\overline{(ab1)(ab2)}$ in $X'$ is assigned $0$ in the cochain as the cyclic orderings at $a12$ in $\Lk{(ab1, \tConf{5}))}$ and in $\Lk{(ab2, \tConf{5})}$ are opposite.}
	\label{fig:theta5-orderingscochain-ab1-a12-ab2}
\end{figure}

We claim it suffices to construct embeddings for vertices $abi$ for all $i \in \{1, 2, 3, 4, 5\}$ so that $\theta_{\phi_{abi}}(aij)$ and $\theta_{\phi_{abj}}(aij)$ are opposite. Indeed, for any embedding $\phi_{aij}$ of $\Lk(aij, \tConf{5})$, the clockwise orientations of $abi$ and $abj$, denoted by $\theta_{\phi_{aij}}(abi)$ and $\theta_{\phi_{aij}}(abj)$, are opposite for $i \neq j$ (for example, see Figure~\ref{fig:theta5-links-a12}). To get $\omega_\Phi (\overline{(abi)(aij)}) = 0$ and $\omega_\Phi (\overline{(abj)(aij)}) = 0$ (see Figure~\ref{fig:theta5-orderingscochain-ab1-a12-ab2} for an explicit example involving the vertices $a12$, $ab1$, and $ab2$), the orientation $\theta_{\phi_{abi}}(aij)$ should be opposite to $\theta_{\phi_{aij}}(abi)$ and the orientation $\theta_{\phi_{abj}}(aij)$ should be opposite to $\theta_{\phi_{aij}}(abj)$ (see Figures~\ref{fig:theta5-links-a12},~\ref{fig:theta5-links-ab1}, and~\ref{fig:theta5-links-ab2}). Thus, we need the clockwise orientations $\theta_{\phi_{abi}}(aij)$ and $\theta_{\phi_{abj}}(aij)$ of $aij$ in $\Lk(abi,\tConf{5})$ and $\Lk(abj,\tConf{5})$, respectively, to be opposite. 
By symmetry, the same is true for the vertex $bij$.
Each edge of $X'$ is of the type $\overline{(abi)(aij)}$ or $\overline{(abj)(bij)}$. Therefore, to get $\omega_\Phi(e) = 0$ for each edge $e$ of $X'$, we need to find a family of embeddings $\Phi = \{\phi_v : \Lk(v, \tConf{5}) \longrightarrow \mathbb{R}^2\ | \text{ } v = abi \text{ for } 1 \leq i \leq 5\}$ such that the cyclic orderings $\theta_{\phi_{abi}}(aij)$ and $\theta_{\phi_{abi}}(aij)$ are opposite for all $1 \leq i,j \leq 5$.

\begin{figure}[!t]
	\centering
	\begin{subfigure}[!t]{0.3\textwidth}
		\centering
		\begin{tikzpicture}[scale=0.5]
			\node[regular polygon, regular polygon sides=4, draw, minimum size=3.9cm] (osquare) {};
			\node[regular polygon, regular polygon sides=4, draw, minimum size=1.5cm] (isquare) {};
            
			\draw (osquare.corner 1) -- (isquare.corner 1);
			\draw (osquare.corner 2) -- (isquare.corner 2);
			\draw (osquare.corner 3) -- (isquare.corner 3);
			\draw (osquare.corner 4) -- (isquare.corner 4);
            
			\filldraw (osquare.corner 1) circle (2.5pt) node[anchor=south] {$b15$};
			\filldraw (osquare.corner 2) circle (2.5pt) node[anchor=south] {$a12$};
			\filldraw (osquare.corner 3) circle (2.5pt) node[anchor=north] {$b13$};
			\filldraw (osquare.corner 4) circle (2.5pt) node[anchor=north] {$a14$};
			\filldraw (isquare.corner 1) circle (2.5pt) node[anchor=west] {$a13$};
			\filldraw (isquare.corner 2) circle (2.5pt) node[anchor=east] {$b14$};
			\filldraw (isquare.corner 3) circle (2.5pt) node[anchor=east] {$a15$};
			\filldraw (isquare.corner 4) circle (2.5pt) node[anchor=west] {$b12$};
		\end{tikzpicture}
		\caption{$\Lk{(ab1)}$}
		\label{fig:theta5-embedding-ab1}
	\end{subfigure}
	\hfill
	\begin{subfigure}[!t]{0.3\textwidth}
		\centering
		\begin{tikzpicture}[scale=0.5]
			\node[regular polygon, regular polygon sides=4, draw, minimum size=3.9cm] (osquare) {};
			\node[regular polygon, regular polygon sides=4, draw, minimum size=1.5cm] (isquare) {};
            
			\draw (osquare.corner 1) -- (isquare.corner 1);
			\draw (osquare.corner 2) -- (isquare.corner 2);
			\draw (osquare.corner 3) -- (isquare.corner 3);
			\draw (osquare.corner 4) -- (isquare.corner 4);
            
			\filldraw (osquare.corner 1) circle (2.5pt) node[anchor=south] {$b24$};
			\filldraw (osquare.corner 2) circle (2.5pt) node[anchor=south] {$a23$};
			\filldraw (osquare.corner 3) circle (2.5pt) node[anchor=north] {$b25$};
			\filldraw (osquare.corner 4) circle (2.5pt) node[anchor=north] {$a12$};
			\filldraw (isquare.corner 1) circle (2.5pt) node[anchor=west] {$a25$};
			\filldraw (isquare.corner 2) circle (2.5pt) node[anchor=east] {$b12$};
			\filldraw (isquare.corner 3) circle (2.5pt) node[anchor=east] {$a24$};
			\filldraw (isquare.corner 4) circle (2.5pt) node[anchor=west] {$b23$};
		\end{tikzpicture}
		\caption{$\Lk{(ab2)}$}
		\label{fig:theta5-embedding-ab2}
	\end{subfigure}
	\hfill
	\begin{subfigure}[!t]{0.3\textwidth}
		\centering
		\begin{tikzpicture}[scale=0.5]
			\node[regular polygon, regular polygon sides=4, draw, minimum size=3.9cm] (osquare) {};
			\node[regular polygon, regular polygon sides=4, draw, minimum size=1.5cm] (isquare) {};
            
			\draw (osquare.corner 1) -- (isquare.corner 1);
			\draw (osquare.corner 2) -- (isquare.corner 2);
			\draw (osquare.corner 3) -- (isquare.corner 3);
			\draw (osquare.corner 4) -- (isquare.corner 4);
            
			\filldraw (osquare.corner 1) circle (2.5pt) node[anchor=south] {$b35$};
			\filldraw (osquare.corner 2) circle (2.5pt) node[anchor=south] {$a31$};
			\filldraw (osquare.corner 3) circle (2.5pt) node[anchor=north] {$b32$};
			\filldraw (osquare.corner 4) circle (2.5pt) node[anchor=north] {$a34$};
			\filldraw (isquare.corner 1) circle (2.5pt) node[anchor=west] {$a32$};
			\filldraw (isquare.corner 2) circle (2.5pt) node[anchor=east] {$b34$};
			\filldraw (isquare.corner 3) circle (2.5pt) node[anchor=east] {$a35$};
			\filldraw (isquare.corner 4) circle (2.5pt) node[anchor=west] {$b31$};
		\end{tikzpicture}
		\caption{$\Lk{(ab3)}$}
		\label{fig:theta5-embedding-ab3}
	\end{subfigure}

	\begin{subfigure}[!t]{0.4\textwidth}
		\centering
		\begin{tikzpicture}[scale=0.5]
			\node[regular polygon, regular polygon sides=4, draw, minimum size=3.9cm] (osquare) {};
			\node[regular polygon, regular polygon sides=4, draw, minimum size=1.5cm] (isquare) {};
            
			\draw (osquare.corner 1) -- (isquare.corner 1);
			\draw (osquare.corner 2) -- (isquare.corner 2);
			\draw (osquare.corner 3) -- (isquare.corner 3);
			\draw (osquare.corner 4) -- (isquare.corner 4);
            
			\filldraw (osquare.corner 1) circle (2.5pt) node[anchor=south] {$b43$};
			\filldraw (osquare.corner 2) circle (2.5pt) node[anchor=south] {$a41$};
			\filldraw (osquare.corner 3) circle (2.5pt) node[anchor=north] {$b42$};
			\filldraw (osquare.corner 4) circle (2.5pt) node[anchor=north] {$a45$};
			\filldraw (isquare.corner 1) circle (2.5pt) node[anchor=west] {$a25$};
			\filldraw (isquare.corner 2) circle (2.5pt) node[anchor=east] {$b45$};
			\filldraw (isquare.corner 3) circle (2.5pt) node[anchor=east] {$a34$};
			\filldraw (isquare.corner 4) circle (2.5pt) node[anchor=west] {$b41$};
		\end{tikzpicture}
		\caption{$\Lk{(ab4)}$}
		\label{fig:theta5-embedding-ab4}
	\end{subfigure}
    \hspace{0.3cm}
	\begin{subfigure}[!t]{0.4\textwidth}
		\centering
		\begin{tikzpicture}[scale=0.5]
			\node[regular polygon, regular polygon sides=4, draw, minimum size=3.9cm] (osquare) {};
			\node[regular polygon, regular polygon sides=4, draw, minimum size=1.5cm] (isquare) {};
            
			\draw (osquare.corner 1) -- (isquare.corner 1);
			\draw (osquare.corner 2) -- (isquare.corner 2);
			\draw (osquare.corner 3) -- (isquare.corner 3);
			\draw (osquare.corner 4) -- (isquare.corner 4);
            
			\filldraw (osquare.corner 1) circle (2.5pt) node[anchor=south] {$b54$};
			\filldraw (osquare.corner 2) circle (2.5pt) node[anchor=south] {$a51$};
			\filldraw (osquare.corner 3) circle (2.5pt) node[anchor=north] {$b52$};
			\filldraw (osquare.corner 4) circle (2.5pt) node[anchor=north] {$a53$};
			\filldraw (isquare.corner 1) circle (2.5pt) node[anchor=west] {$a52$};
			\filldraw (isquare.corner 2) circle (2.5pt) node[anchor=east] {$b53$};
			\filldraw (isquare.corner 3) circle (2.5pt) node[anchor=east] {$a54$};
			\filldraw (isquare.corner 4) circle (2.5pt) node[anchor=west] {$b51$};
		\end{tikzpicture}
		\caption{$\Lk{(ab5)}$}
		\label{fig:theta5-embedding-ab5}
	\end{subfigure}
	\caption{A family of embeddings of $\Lk(abi,\tConf{5})$ into $\mathbb{R}^2$ so that the associated cochain is trivial.}
	\label{fig:theta5-explicitembedding}
\end{figure}
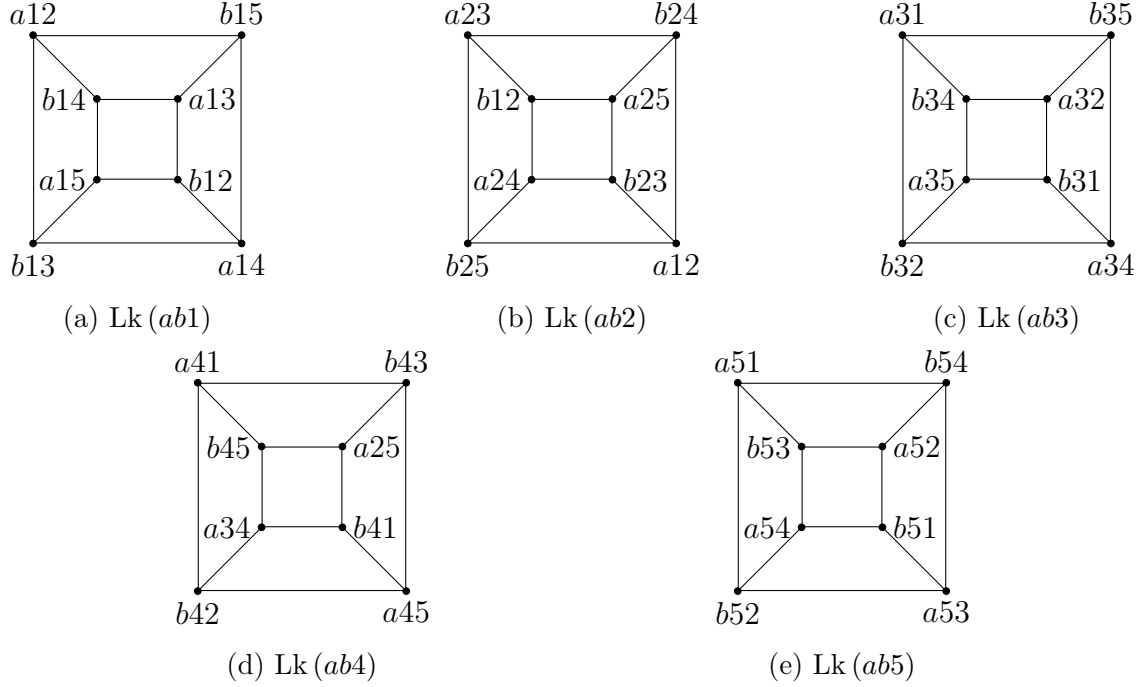

One such family of embeddings that satisfies this requirement is shown in Figure~\ref{fig:theta5-explicitembedding}. The cyclic orderings of each vertex of type $aij$ (and $bij$) in any two links $\Lk{(abi, \tConf{5}))}$ and $\Lk{(abj, \tConf{5})})$ are opposite; therefore, all edges in the underlying graph $X'$ are assigned zero in the cocycle. 
\end{proof}

\begin{remark}
    Given one embedding of $\Lk(ab1, \tConf{5})$ as in Figure~\ref{fig:theta5-explicitembedding}, any element of the symmetric group $S_5$ which takes $1$ to $i$ gives a new embedding of $\Lk(abi, \tConf{5})$. The embeddings that we constructed were obtained by repeatedly applying the cyclic permutation $(12345)$. We believe that applying any element of the alternating group $A_5$ should produce an embedding of $\Lk(abi, \tConf{5})$ so that the resulting Lasheras' cocycle is trivial. 
\end{remark}


\section{\texorpdfstring{$B_3(\Theta_7)$ is not a $3$-manifold group}{B3(Theta7) is not a 3-manifold group}}

\noindent
In this section, we show that $B_3(\Theta_7)$ is not the fundamental group of a $3$-manifold. In fact, we show this for any group quasi-isometric to $B_3(\Theta_7)$; in particular, $B_3(\Theta_7)$ is not virtually a $3$-manifold group. To do so, we will find an obstruction to the group being a $3$-manifold in its boundary, which is defined as follows.

\begin{Definition}
    The boundary of a CAT($0$) space $X$, denoted by $\partial_\infty X$, is the set of equivalence classes of geodesic rays where two rays $\alpha(t)$ and $\beta(t)$ are equivalent if and only if there is a constant $k$ such that $d(\alpha(t),\beta(t))\leq k$ for all $t \ge 0$. 
    The boundary $\partial_\infty X$ carries a natural topology, the \textit{cone topology}, with respect to which $\overline{X} = X \cup \partial_\infty X$ is compact when $X$ is proper (see~\cite[Part~II.8]{MR1744486} for more details).
    
    The visual boundary of a CAT($0$) group is defined as the boundary of a CAT($0$) space on which it acts properly discontinuously and cocompactly. This is well-defined for hyperbolic CAT($0$) groups and is invariant under quasi-isometry. 
\end{Definition}

\noindent Theorem~\ref{thm:main2} then follows from the following theorem adapted from the work of Bestvina--Kapovich--Kleiner~\cite[Theorem~1 and Corollary~19]{MR1933584}. 

\begin{Theorem}
    Let $G$ be a hyperbolic CAT($0$) group with visual boundary $\partial_\infty G$. If there is a non-planar graph embedded in $\partial_\infty G$, then $G$ is not a $3$-manifold group.
\end{Theorem}

\noindent
We find an embedded $K_{3,3}$ in the visual boundary $ \partial_\infty(B_3(\Theta_7))$. The following sequence of lemmas identify the quasi-convex subgroups of $B_3(\Theta_7)$ that give rise to this embedded $K_{3,3}$.

We will use the following result from~\cite[Theorem~1(2)]{MR2077673} on the links of the vertices of a subcomplex of a locally CAT($0$) cube complex to show the existence of such subgroups in $B_3(\Theta_7)$.

\begin{Lemma}
Let $A$ be a finite dimensional cube complex that is locally CAT($0$) and $B$ be a subcomplex of $A$. Then $B$ is a locally CAT($0$), $\pi_1$-injective subcomplex that lifts to an isometric embedding of $\widetilde{B}$ in $\widetilde{A} \iff \Lk(v,B)$ is a full subcomplex in $\Lk(v,A)$ for every vertex $v$ in $B$. 
\end{Lemma}
\noindent
As observed in~\cite[Corollary~3.14]{MR2701024} and ~\cite[Proposition~3.10]{MR4227231}, for configuration spaces of graphs, this translates to the following lemma.

\begin{Lemma}\label{lem:embedding}
    Let $\Gamma'$ be a subgraph of a connected graph $\Gamma$. Then the natural inclusion map $\cConf_n(\Gamma') \hookrightarrow \cConf_n(\Gamma)$ is $\pi_1$-injective. Moreover, the inclusion lifts to an isometric embedding of $\widetilde{\cConf_n{\Gamma'}}$ into $\widetilde{\cConf_n{\Gamma}}$. 
\end{Lemma}

\begin{proof}
    To show that the inclusion map $\cConf_n(\Gamma') \hookrightarrow \cConf_n(\Gamma)$ is $\pi_1$-injective, it is enough to show that $\Lk(v,\cConf_n(\Gamma'))$ is full as a subcomplex of $\Lk(v,\cConf_n(\Gamma))$. Note that an $m$-simplex in the $\Lk(v,\cConf_n(\Gamma))$ where $v=\{v_1,v_2,\dots, v_n\}$ with each $v_i$ a vertex in $\Gamma$ corresponds to a collection of $m$ pairwise disjoint edges in $\Gamma$, each containing exactly one $v_i$. Without changing the homotopy type of $\cConf_n(\Gamma)$, we can barycentrically subdivide the edges of $\Gamma$ not in $\Gamma'$ so that $\Gamma'$ is an induced subgraph of $\Gamma$. A collection of edges disjoint in $\Gamma$, if they exist in $\Gamma'$, are also disjoint in $\Gamma'$. Therefore, if there is an $m$-simplex $C$ in $\Lk(v,\cConf_n(\Gamma))$ such that all its vertices are in $\Lk(v,\cConf_n(\Gamma'))$, then $C$ is in $\Lk(v,\cConf_n(\Gamma'))$.
\end{proof}

\begin{Corollary}\label{subgroup}
For $n\leq m$, the natural inclusion map $\tConf{n} \hookrightarrow \tConf{m}$ is $\pi_1$-injective that lifts to an isometric embedding of the universal covers $\widetilde{\tConf{n}}$ into $\widetilde{\tConf{m}}$. In particular, $B_3(\Theta_n)$ embeds into $B_3(\Theta_m)$. Moreover, this is a quasi-convex embedding which extends to a toplogical embedding of $\partial_\infty \widetilde{ \tConf{n}} \hookrightarrow \partial_\infty \widetilde{\tConf{m}}$ for $n\leq m$.
\end{Corollary}

\noindent The above corollary is a direct result of Lemma~\ref{lem:embedding} and the fact that a quasi-convex subgroup $H$ of a hyperbolic group $G$ leads to a topological embedding of the boundary $\partial_\infty H$ into $\partial_\infty G$~\cite[Section~5.3]{MR919829}.

For the cube complex $\tConf{7}$, this means that we can find $\pi_1$-injective subcomplexes in $\tConf{7}$ by taking subgraphs of $\Theta_7$. In particular, for every choice of two vertices $\tau_1,\tau_2 \in \{1,2, \dots, 7\}$  of $\Theta_7$, we have a non-trivial simple closed curve coming from $\tConf{2} \simeq S^1$ in $\tConf{7}$. We will denote this closed curve by $\Gamma(\tau_1\tau_2)$. Similarly, for every choice of four vertices $\lambda_1,\lambda_2,\lambda_3,\lambda_4$ in $\{1,2,\dots, 7\}$ of $\Theta_7$, we have a locally convex surface $\Gamma(\lambda_1\lambda_2\lambda_3\lambda_4) \simeq \Sigma_3$ sitting inside of $\cConf_3(\Theta_7)$. 
Simple closed curves in these subsurfaces in $\cConf_3(\Theta_7)$ lift to bi-infinite geodesics in the respective universal covers, and correspond to points in the boundary circles of the subsurfaces they are contained in, forming a subdivided $K_{3,3}$ embedded in the visual boundary of $B_3(\Theta_7)$. This is shown explicitly in the proof of Theorem~\ref{thm:main2}.

\begin{Theorem}\label{thm:main2}
    $B_3(\Theta_7)$ is not a $3$-manifold group.
\end{Theorem}

\begin{proof}
    
From Corollary~\ref{subgroup}, the subcomplex $\Gamma(1234)$ of $\cConf_3(\Theta_7)$ is $\pi_1$-injective in $\cConf_3(\Theta_7)$. A local picture about the vertex $ab1$ in Figure~\ref{hexagon} shows that there are simple closed curves $\Gamma(12)$, $\Gamma(13)$, and $\Gamma(14)$ intersecting exactly at $ab1$ in $\Gamma(1234)$. Moreover, these curves lift to transverse pairwise intersecting bi-infinite geodesics in the universal cover $\widetilde{\Gamma(1234)}$, intersecting precisely at a lift $\widetilde{ab1}$ of $ab1$.

Each of these bi-infinite geodesics give two infinite geodesic rays emanating from $\widetilde{ab1}$, in the directions defined by $\widetilde{a1j}$ and $\widetilde{b1j}$ for $2\leq j \leq 4$. We will denote the geodesic ray starting from $\widetilde{ab1}$ in the direction of $\widetilde{a1j}$ by $\gamma_{a1j}$ and starting from $\widetilde{ab1}$ in the direction of $\widetilde{b1j}$ by $\gamma_{b1j}$.

\begin{figure}[t]
    \centering
\begin{subfigure}[t]{0.45\textwidth}
    {
    \begin{tikzpicture}[scale=1.4]

\node[draw, circle, fill=black, inner sep=1.5pt, label=below:$ab1$] (ab1) at (0,0) {};
\node[draw, circle, fill=black, inner sep=1.5pt, label=below:$b12$] (b12) at (0.866,0.5) {};
\node[draw, circle, fill=black, inner sep=1.5pt, label=below:$a13$] (a13) at (0,1) {};
\node[draw, circle, fill=black, inner sep=1.5pt, label=below:$b14$] (b14) at (-0.866,0.5) {};
\node[draw, circle, fill=black, inner sep=1.5pt, label=below:$a12$] (a12) at (-0.866,-0.5) {};
\node[draw, circle, fill=black, inner sep=1.5pt, label=below:$b13$] (b13) at (0,-1) {};
\node[draw, circle, fill=black, inner sep=1.5pt, label=below:$a14$] (a14) at (0.866,-0.5) {};


\node[draw, circle, fill=black, inner sep=1.5pt, label=above:$123$] (123r) at (0.866,1.5) {};
\node[draw, circle, fill=black, inner sep=1.5pt, label=above:$234$] (234l) at (-0.866,1.5) {};
\node[draw, circle, fill=black, inner sep=1.5pt, label=below:$124$] (124l) at (-1.732,0) {};
\node[draw, circle, fill=black, inner sep=1.5pt, label=below:$123$] (123l) at (-0.866,-1.5) {};
\node[draw, circle, fill=black, inner sep=1.5pt, label=below:$234$] (234r) at (0.866,-1.5) {};
\node[draw, circle, fill=black, inner sep=1.5pt, label=below:$124$] (124r) at (1.732,0) {};


\draw (b14) -- (124l);
\draw (b14) -- (234l);
\draw (a12) -- (124l);
\draw (a12) -- (123l);
\draw (b13) -- (123l);
\draw (b13) -- (234r);
\draw (a14) -- (234r);
\draw (a14) -- (124r);
\draw (b12) -- (124r);
\draw (b12) -- (123r);
\draw (a13) -- (234l);
\draw (a13) -- (123r);

\node[draw, circle, fill=black, inner sep=1.5pt, label=below:$ab2$] (b12o) at (1.732,1) {};
\node[draw, circle, fill=black, inner sep=1.5pt, label=above:$ab3$] (a13o) at (0,2) {};
\node[draw, circle, fill=black, inner sep=1.5pt, label=below:$ab4$] (b14o) at (-1.732,1) {};
\node[draw, circle, fill=black, inner sep=1.5pt, label=below:$ab2$] (a12o) at (-1.732,-1) {};
\node[draw, circle, fill=black, inner sep=1.5pt, label=below:$ab3$] (b13o) at (0,-2) {};
\node[draw, circle, fill=black, inner sep=1.5pt, label=below:$ab4$] (a14o) at (1.732,-1) {};


\draw[green!50!black] (ab1) -- (a13) -- (a13o);
\draw[green!50!black] (ab1) -- (b13) -- (b13o);
\draw[red] (ab1) -- (a12) -- (a12o); 
\draw[red] (ab1) -- (b12) -- (b12o);
\draw[blue] (ab1) -- (a14) -- (a14o);
\draw[blue] (ab1) -- (b14) -- (b14o);

\end{tikzpicture}
}
\end{subfigure}
\hfill
\begin{subfigure}[t]{0.5\textwidth}
{
    \begin{tikzpicture}[scale=1.2]

\node[draw, circle, fill=black, inner sep=1pt, label={[xshift=-2pt]left:{$\scriptstyle \widetilde{ab1}$}}] (ab1) at (0,0) {};

\node[label=below:$\widetilde{\gamma_{b12}}$] (b12o) at (1.732,1) {};
\node[label=below right:$\widetilde{\gamma_{a13}}$] (a13o) at (0,2) {};
\node[label=below:$\widetilde{\gamma_{b14}}$] (b14o) at (-1.732,1) {};
\node[label=below:$\widetilde{\gamma_{a12}}$] (a12o) at (-1.732,-1) {};
\node[label=below:$\widetilde{\gamma_{b13}}$] (b13o) at (0,-2) {};
\node[label=below:$\widetilde{\gamma_{a14}}$] (a14o) at (1.732,-1) {};


\draw[->, thick][green!50!black] (ab1) -- (a13o);
\draw[->, thick][green!50!black] (ab1) -- (b13o);
\draw[->, thick][red] (ab1) -- (a12o); 
\draw[->, thick][red] (ab1) -- (b12o);
\draw[->, thick][blue] (ab1) -- (a14o);
\draw[->, thick][blue] (ab1)  -- (b14o);

\draw (0,0) circle [radius=3cm];
\node[draw, circle, fill=red, inner sep=1.5pt] (b12tilde) at (2.598,1.5) {};
\node[draw, circle, fill=red, inner sep=1.5pt] (a12tilde) at (-2.598,-1.5) {};
\node[draw, circle, fill=green!50!black, inner sep=1.5pt] (b13tilde) at (0,-3) {};
\node[draw, circle, fill=green!50!black, inner sep=1.5pt] (b12tilde) at (0,3) {};
\node[draw, circle, fill=blue, inner sep=1.5pt] (b12tilde) at (-2.598,1.5) {};
\node[draw, circle, fill=blue, inner sep=1.5pt] (b12tilde) at (2.598,-1.5) {};

\end{tikzpicture}
}
\end{subfigure}
\caption{ (Left) Simple closed curves of $\cConf_3(\Theta_4)$ embedded in $\cConf_3(\Theta_7)$ and (right) the six geodesic rays corresponding to these closed curves in $\cConf_3(\Theta_7)$ and the corresponding points in $\partial_\infty(\widetilde{\cConf_3(\Theta_7)})$.}

    \label{hexagon}
\end{figure}
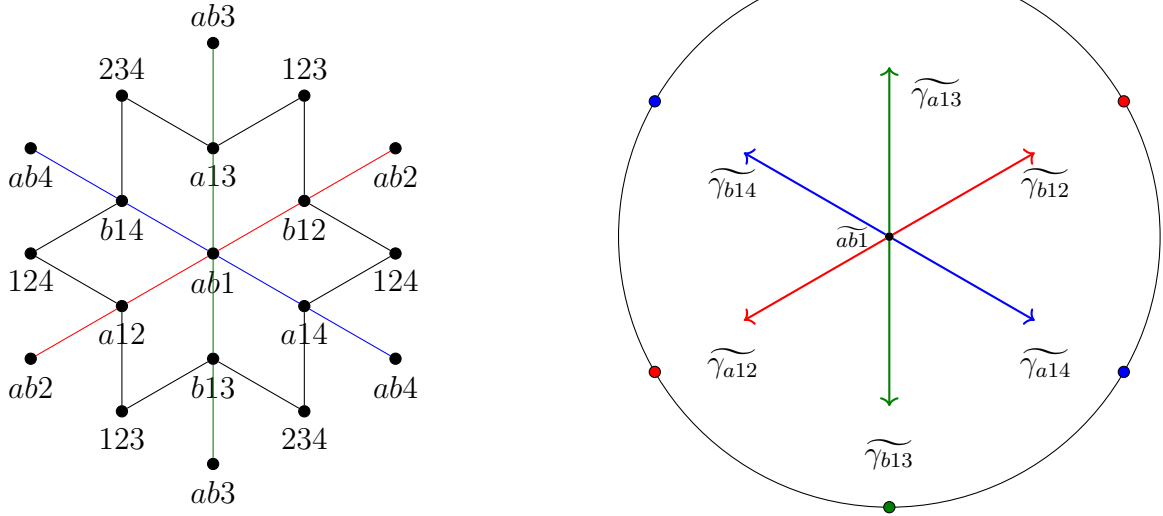

The visual boundary of $\widetilde{\cConf_3(\Theta_4)}$ is homeomorphic to $S^1$. Since the three bi-infinite geodesics corresponding to the lifts of $\Gamma(12), \Gamma(13)$, and $\Gamma(14)$ intersect once transversely in $\widetilde{\Gamma(1234)}$, each of the six associated geodesic rays gives exactly one point in the visual boundary $\partial_\infty \widetilde{\cConf_3(\Theta_4)}$ of $\widetilde{\cConf_3(\Theta_4)}$. We will denote these points by the points in $\widetilde{\Gamma(1234)}$ that define the direction of these rays in the universal cover. So, the six points on the boundary $S^1$ are $\widetilde{a12}, \widetilde{b12}, \widetilde{a13}, \widetilde{b13}, \widetilde{a14}, \text{ and } \widetilde{b14}$. 

These six points lie in the cyclic order $\widetilde{a12}, \widetilde{b13}, \widetilde{a14}, \widetilde{b12}, \widetilde{a13},  \text{ and } \widetilde{b14}$ on the boundary circle of $\Gamma(1234)$, as shown in Figure~\ref{hexagon}. This can be seen by first considering the half-plane of $\widetilde{\Gamma(1234)}$ containing the geodesic rays $\gamma_{a12}$, $\gamma_{b13}$ and $\gamma_{a14}$ to get the order  $\widetilde{a12}, \widetilde{b13}, \text{ and } \widetilde{a14}$ of the vertices. Now consider the half-plane of $\widetilde{\Gamma(1234)}$ containing the geodesic rays $\gamma_{b12}$, $\gamma_{a13}$ and $\gamma_{b14}$, to get the order $\widetilde{b12}, \widetilde{a13}, \text{ and }  \widetilde{b14}$. These six vertices, along with the arcs of the boundary circle, form a hexagon in $\partial_\infty \widetilde{\cConf_3(\Theta_4)}$.  
\begin{figure}[t]
\begin{minipage}[c]{.55\textwidth}
\adjustbox{width=\linewidth}{%
        
        \begin{tikzpicture}
        
        \begin{scope}[xscale=1.3, yscale=2]
		\tikzset{every picture/.style={line width=1pt}} 
            \filldraw[black] (0, 1) circle (1pt) node[anchor=south]{$b$};
			\filldraw[black] (0, -1) circle (1pt) node[anchor=north]{$a$};
			\filldraw[black] (-1.5, 0) circle (1pt) node[anchor=west]{$1$};
			\filldraw[black] (-1, 0) circle (1pt) node[anchor=west]{$2$};
			\filldraw[black] (-0.5, 0) circle (1pt) node[anchor=west]{$3$};
			\filldraw[black] (0, 0) circle (1pt) node[anchor=west]{$4$};
			\filldraw[black] (0.5, 0) circle (1pt) node[anchor=west]{$5$};
			\filldraw[black] (1, 0) circle (1pt) node[anchor=west]{$6$};
			\filldraw[black] (1.5, 0) circle (1pt) node[anchor=west]{$7$};
			
			\draw (0, 1) to[out=180, in=90] (-1.5, 0) to[out=-90, in=180] (0, -1);
            \draw[line width=6pt, yellow!50, opacity=0.3, line cap=round] (0, 1) to[out=180, in=90] (-1.5, 0) to[out=-90, in=180] (0, -1);
            
			\draw (0, 1) to[out=190, in=90] (-1, 0) to[out=-90, in=170] (0, -1);
            \draw[line width=6pt, yellow!50, opacity=0.3, line cap=round] (0, 1) to[out=190, in=90] (-1, 0) to[out=-90, in=170] (0, -1);
            
			\draw (0, 1) to[out=230, in=90] (-0.5, 0) to[out=-90, in=130] (0, -1);
            \draw[line width=6pt, yellow!50, opacity=0.3, line cap=round] (0, 1) to[out=230, in=90] (-0.5, 0) to[out=-90, in=130] (0, -1);
            
			\draw (0, 1) to[out=-90, in=90] (0, 0) to[out=-90, in=90] (0, -1);
            \draw[line width=6pt, yellow!50, opacity=0.3, line cap=round] (0, 1) to[out=-90, in=90] (0, 0) to[out=-90, in=90] (0, -1);
            
			\draw (0, 1) to[out=-30, in=90] (0.5, 0) to[out=-90, in=30] (0, -1);
			\draw (0, 1) to[out=-10, in=90] (1, 0) to[out=-90, in=10] (0, -1);
			\draw (0, 1) to[out=0, in=90] (1.5, 0) to[out=-90, in=0] (0, -1);

            \begin{scope}[xscale=0.77, yscale=0.5]
            \draw (0,0) circle [radius=3.4cm];
            \node[draw, circle, fill=green!50!black, inner sep=1.5pt, label=below:$\widetilde{b12}$] (b12tilde) at (3,1.6) {};
            \node[draw, circle, fill=green!50!black, inner sep=1.5pt, label=below:$\widetilde{a12}$] (a12tilde) at (-3,-1.6) {};
            \node[draw, circle, fill=blue, inner sep=1.5pt, label=below:$\widetilde{b13}$] (b13tilde) at (0,-3.4) {};
            \node[draw, circle, fill=blue, inner sep=1.5pt, label=below:$\widetilde{a13}$] (b12tilde) at (0,3.4) {};
            \node[draw, circle, fill=red, inner sep=1.5pt, label=below:$\widetilde{b14}$] (b12tilde) at (-3,1.6) {};
            \node[draw, circle, fill=red, inner sep=1.5pt, label=below:$\widetilde{a14}$] (b12tilde) at (3,-1.6) {};

            \node[draw=none, fill=none, inner sep=0pt] at (1.4,-2.3) {$\Gamma(1234)$};

            \end{scope}
            
              \begin{scope}[shift={(3.5,2.5)}, scale=0.75]

                \filldraw[black] (0, 1) circle (1pt) node[anchor=south]{$b$};
    			\filldraw[black] (0, -1) circle (1pt) node[anchor=north]{$a$};
    			\filldraw[black] (-1.5, 0) circle (1pt) node[anchor=west]{$1$};
    			\filldraw[black] (-1, 0) circle (1pt) node[anchor=west]{$2$};
    			\filldraw[black] (-0.5, 0) circle (1pt) node[anchor=west]{$3$};
    			\filldraw[black] (0, 0) circle (1pt) node[anchor=west]{$4$};
    			\filldraw[black] (0.5, 0) circle (1pt) node[anchor=west]{$5$};
    			\filldraw[black] (1, 0) circle (1pt) node[anchor=west]{$6$};
    			\filldraw[black] (1.5, 0) circle (1pt) node[anchor=west]{$7$};

    			\draw (0, 1) to[out=180, in=90] (-1.5, 0) to[out=-90, in=180] (0, -1);
                \draw[line width=6pt, green!50!black!50, opacity=0.3, line cap=round] (0, 1) to[out=180, in=90] (-1.5, 0) to[out=-90, in=180] (0, -1);

    			\draw (0, 1) to[out=190, in=90] (-1, 0) to[out=-90, in=170] (0, -1);
                
    			\draw (0, 1) to[out=230, in=90] (-0.5, 0) to[out=-90, in=130] (0, -1);
                \draw[line width=6pt, green!50!black!50, opacity=0.3, line cap=round] (0, 1) to[out=230, in=90] (-0.5, 0) to[out=-90, in=130] (0, -1);
                
    			\draw (0, 1) to[out=-90, in=90] (0, 0) to[out=-90, in=90] (0, -1);
    			\draw (0, 1) to[out=-30, in=90] (0.5, 0) to[out=-90, in=30] (0, -1);
    			\draw (0, 1) to[out=-10, in=90] (1, 0) to[out=-90, in=10] (0, -1);
                \draw[line width=6pt, green!50!black!50, opacity=0.3, line cap=round] (0, 1) to[out=-10, in=90] (1, 0) to[out=-90, in=10] (0, -1);
                
    			\draw (0, 1) to[out=0, in=90] (1.5, 0) to[out=-90, in=0] (0, -1);
                \draw[line width=6pt, green!50!black!50, opacity=0.3, line cap=round] (0, 1) to[out=0, in=90] (1.5, 0) to[out=-90, in=0] (0, -1);

                \node[draw=none, fill=none, inner sep=0pt] at (1.9,-1) {$\Gamma(1367)$};

            \end{scope}

              \begin{scope}[shift={(-3.5,2.5)}, scale=0.75]

                \filldraw[black] (0, 1) circle (1pt) node[anchor=south]{$b$};
    			\filldraw[black] (0, -1) circle (1pt) node[anchor=north]{$a$};
    			\filldraw[black] (-1.5, 0) circle (1pt) node[anchor=west]{$1$};
    			\filldraw[black] (-1, 0) circle (1pt) node[anchor=west]{$2$};
    			\filldraw[black] (-0.5, 0) circle (1pt) node[anchor=west]{$3$};
    			\filldraw[black] (0, 0) circle (1pt) node[anchor=west]{$4$};
    			\filldraw[black] (0.5, 0) circle (1pt) node[anchor=west]{$5$};
    			\filldraw[black] (1, 0) circle (1pt) node[anchor=west]{$6$};
    			\filldraw[black] (1.5, 0) circle (1pt) node[anchor=west]{$7$};
    			
    			\draw (0, 1) to[out=180, in=90] (-1.5, 0) to[out=-90, in=180] (0, -1);
                \draw[line width=6pt, red!50, opacity=0.3, line cap=round] (0, 1) to[out=180, in=90] (-1.5, 0) to[out=-90, in=180] (0, -1);
                
    			\draw (0, 1) to[out=190, in=90] (-1, 0) to[out=-90, in=170] (0, -1);
                \draw[line width=6pt, red!50, opacity=0.3, line cap=round] (0, 1) to[out=190, in=90] (-1, 0) to[out=-90, in=170] (0, -1);
                
    			\draw (0, 1) to[out=230, in=90] (-0.5, 0) to[out=-90, in=130] (0, -1);
    			\draw (0, 1) to[out=-90, in=90] (0, 0) to[out=-90, in=90] (0, -1);
    			\draw (0, 1) to[out=-30, in=90] (0.5, 0) to[out=-90, in=30] (0, -1);
                \draw[line width=6pt, red!50, opacity=0.3, line cap=round] (0, 1) to[out=-30, in=90] (0.5, 0) to[out=-90, in=30] (0, -1);
                
    			\draw (0, 1) to[out=-10, in=90] (1, 0) to[out=-90, in=10] (0, -1);
                \draw[line width=6pt, red!50, opacity=0.3, line cap=round] (0, 1) to[out=-10, in=90] (1, 0) to[out=-90, in=10] (0, -1);
                
    			\draw (0, 1) to[out=0, in=90] (1.5, 0) to[out=-90, in=0] (0, -1);

                \node[draw=none, fill=none, inner sep=0pt] at (1.9,-1) {$\Gamma(1256)$};

            \end{scope}

              \begin{scope}[shift={(0,-3.4)}, scale=0.75]

                \filldraw[black] (0, 1) circle (1pt) node[anchor=south]{$b$};
    			\filldraw[black] (0, -1) circle (1pt) node[anchor=north]{$a$};
    			\filldraw[black] (-1.5, 0) circle (1pt) node[anchor=west]{$1$};
    			\filldraw[black] (-1, 0) circle (1pt) node[anchor=west]{$2$};
    			\filldraw[black] (-0.5, 0) circle (1pt) node[anchor=west]{$3$};
    			\filldraw[black] (0, 0) circle (1pt) node[anchor=west]{$4$};
    			\filldraw[black] (0.5, 0) circle (1pt) node[anchor=west]{$5$};
    			\filldraw[black] (1, 0) circle (1pt) node[anchor=west]{$6$};
    			\filldraw[black] (1.5, 0) circle (1pt) node[anchor=west]{$7$};
    			
    			\draw (0, 1) to[out=180, in=90] (-1.5, 0) to[out=-90, in=180] (0, -1);
                \draw[line width=6pt, blue!50, opacity=0.3, line cap=round] (0, 1) to[out=180, in=90] (-1.5, 0) to[out=-90, in=180] (0, -1);
                
    			\draw (0, 1) to[out=190, in=90] (-1, 0) to[out=-90, in=170] (0, -1);
    			\draw (0, 1) to[out=230, in=90] (-0.5, 0) to[out=-90, in=130] (0, -1);
    			\draw (0, 1) to[out=-90, in=90] (0, 0) to[out=-90, in=90] (0, -1);
                \draw[line width=6pt, blue!50, opacity=0.3, line cap=round] (0, 1) to[out=-90, in=90] (0, 0) to[out=-90, in=90] (0, -1);
                
    			\draw (0, 1) to[out=-30, in=90] (0.5, 0) to[out=-90, in=30] (0, -1);
                \draw[line width=6pt, blue!50, opacity=0.3, line cap=round] (0, 1) to[out=-30, in=90] (0.5, 0) to[out=-90, in=30] (0, -1);
                
    			\draw (0, 1) to[out=-10, in=90] (1, 0) to[out=-90, in=10] (0, -1);
    			\draw (0, 1) to[out=0, in=90] (1.5, 0) to[out=-90, in=0] (0, -1);
                \draw[line width=6pt, blue!50, opacity=0.3, line cap=round] (0, 1) to[out=0, in=90] (1.5, 0) to[out=-90, in=0] (0, -1);
                \node[draw=none, fill=none, inner sep=0pt] at (1.9,-1) {$\Gamma(1457)$};
             \end{scope}
             \end{scope}
		\end{tikzpicture}
}

\end{minipage}\hfill
\begin{minipage}[c]{.40\textwidth}
\adjustbox{width=\linewidth}{%

\begin{tikzpicture}
    
    \node[draw, circle, fill=black, inner sep=1.5pt, label=below:$\widetilde{a12}$] (a12) at (0,6) {};
    \node[draw, circle, fill=red, inner sep=1.5pt, label=below:$\widetilde{b15}$] (b15) at (4,6) {};
    \node[draw, circle, fill=red, inner sep=1.5pt, label=below:$\widetilde{a16}$] (a16) at (8,6) {};
    \node[draw, circle, fill=black, inner sep=1.5pt, label=below:$\widetilde{b12}$] (b12) at (12,6) {};

    \node[draw, circle, fill=black, inner sep=1.5pt, label=below:$\widetilde{a13}$] (a13) at (0,3) {};
    \node[draw, circle, fill=green!50!black, inner sep=1.5pt, label=below:$\widetilde{b16}$] (b16) at (4,3) {};
    \node[draw, circle, fill=green!50!black, inner sep=1.5pt, label=below:$\widetilde{a17}$] (a17) at (8,3) {};
    \node[draw, circle, fill=black, inner sep=1.5pt, label=below:$\widetilde{b13}$] (b13) at (12,3) {};

    \node[draw, circle, fill=black, inner sep=1.5pt, label=below:$\widetilde{a14}$] (a14) at (0,0) {};
    \node[draw, circle, fill=blue, inner sep=1.5pt, label=below:$\widetilde{b17}$] (b17) at (4,0) {};
    \node[draw, circle, fill=blue, inner sep=1.5pt, label=below:$\widetilde{a15}$] (a15) at (8,0) {};
    \node[draw, circle, fill=black, inner sep=1.5pt, label=below:$\widetilde{b14}$] (b14) at (12,0) {};

    \draw[red] (a12) -- (b15) --node[above] {$\partial_{\infty} \widetilde{\Gamma(1256)}$} (a16) -- (b12);
    \draw[green!50!black] (a13) -- (b16) --node[above] {$\partial_{\infty} \widetilde{\Gamma(1367)}$} (a17) -- (b13);
    \draw[blue] (a14) -- (b17) --node[above] {$\partial_{\infty} \widetilde{\Gamma(1457)}$} (a15) -- (b14);

    \draw (a12) -- (b13);
    \draw (a12) -- (b14);
    \draw (a13) -- (b12);
    \draw (a13) -- (b14);
    \draw (a14) -- (b12);
    \draw (a14) -- (b13);

\end{tikzpicture}}

\end{minipage}
\caption{Subsurfaces $\Gamma(1234)$, $\Gamma(1256)$, $\Gamma(1367)$, and $\Gamma(1457)$ and the corresponding $K_{3,3}$ in $\partial_\infty B_3(\Theta_7)$.}
\label{fig:subsurfaces}
\end{figure}
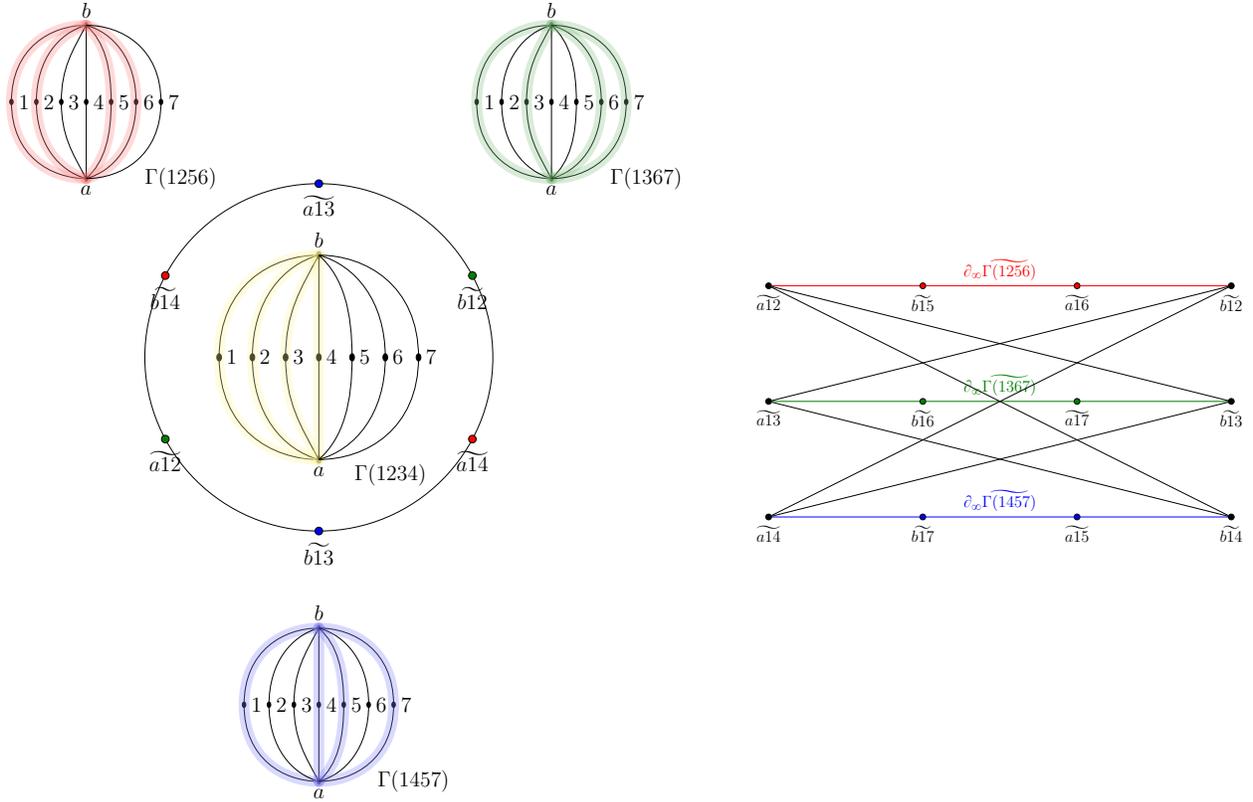

 Now, consider the subsurfaces $\Gamma(1256)$, $\Gamma(1367)$, and $\Gamma(1457)$, as shown in 
 Figure~\ref{fig:subsurfaces}, with universal covers $\widetilde{\Gamma(1256)}$, $\widetilde{\Gamma(1367)}$, and $\widetilde{\Gamma(1457)}$, respectively. Each of these universal covers intersects $\widetilde{\Gamma(1234)}$ in distinct geodesics. These geodesics divide $\widetilde{\Gamma(1256)}$, $\widetilde{\Gamma(1367)}$, and $\widetilde{\Gamma(1457)}$ into half-planes, the visual boundaries of which contain the subdivided edges (in the order) $(\widetilde{a12}, \widetilde{b15}, \widetilde{a16}, \widetilde{b12})$, $(\widetilde{a13}, \widetilde{b16}, \widetilde{a17}, \widetilde{b13})$ and $(\widetilde{a14}, \widetilde{b17}, \widetilde{a15}, \widetilde{b14})$, respectively, thus completing the hexagon to a $K_{3,3}$ in $\partial_\infty \widetilde{\cConf_3(\Gamma_7)}$, as shown in Figure~\ref{fig:subsurfaces}.
\end{proof}

\begin{remark}
    The choice of embedded $\cConf_3(\Theta_4)$, and thus of the corresponding surface groups, is not unique and just depends on choice of four sets of four strands (i.e. paths from vertex $a$ to $b$) in the $\Theta_7$ graph that pairwise intersect in two strands.
\end{remark}

\begin{proof}[Proof of Theorem~\ref{thm:main7}]
    Since quasi-isometric hyperbolic groups have homeomorphic boundaries, it follows that any group quasi-isometric to $B_3(\Theta_7)$ is not a $3$-manifold group. 
\end{proof}

\begin{Corollary}\label{cor:3m}
    By Corollary~\ref{subgroup}, $B_3(\Theta_n)$ is not a $3$-manifold group for any $n \geq 7$. 
 \end{Corollary}

\begin{figure}[t]
\vspace{0pt}
\begin{subfigure}{0.3\textwidth}
    {
\begin{tikzpicture}[xscale = 1, yscale = 1.5]

    \node[draw, circle, fill=black, inner sep=1.5pt, label=left:$123$] (A) at (1,0) {};
    \node[draw, circle, fill=black, inner sep=1.5pt, label={right:$124$}] (B) at (1.75,0) {};
   \node[] (C1) at (2.75,0) {};
   \node[] (C2) at (3.5,0) {};  
   
    \node[draw, circle, fill=black, inner sep=1.5pt, label=right:$n12$] (C) at (4,0) {};
   
    \node[draw, circle, fill=black, inner sep=1.5pt, label=above:$ab1$] (G) at (2.5,1) {};
    \node[draw, circle, fill=black, inner sep=1.5pt, label=below:$ab2$] (H) at (2.5,-1) {};

    \draw[black, thick] (G) -- (A);
    \draw[black, thick] (G) -- (B);
    \draw[black, thick] (G) -- (C);

    \draw[black, thick] (H) -- (A);
    \draw[black, thick] (H) -- (B);
    \draw[black,thick] (H) -- (C);
    
    \draw[dotted] (C1) -- (C2);

    \end{tikzpicture}
}
\subcaption[]{$v = a12$ or $v=b12$}
\end{subfigure}
\hfill
\begin{subfigure}{0.2\textwidth}
{
\begin{tikzpicture}[xscale=2, yscale=1.5]

    \node[draw, circle, fill=black, inner sep=1.5pt, label=above:$a23$] (A1) at (0,1) {};
    \node[draw, circle, fill=black, inner sep=1.5pt, label=above:$b23$] (A2) at (1,1) {};
    \node[draw, circle, fill=black, inner sep=1.5pt, label=left:$a13$] (B1) at (0,0) {};
    \node[draw, circle, fill=black, inner sep=1.5pt, label=right:$b13$] (B2) at (1,0) {};
    \node[draw, circle, fill=black, inner sep=1.5pt, label=below:$a12$] (C1) at (0,-1) {};
    \node[draw, circle, fill=black, inner sep=1.5pt, label=below:$b12$] (C2) at (1,-1) {};

    \draw[black, thick] (A1) -- (B2);
    \draw[black, thick] (A1) -- (C2);   
    \draw[black, thick] (B1) -- (A2);
    \draw[black, thick] (B1) -- (C2);
    \draw[black, thick] (C1) -- (A2);   
    \draw[black, thick] (C1) -- (B2);
    
\end{tikzpicture}
}
\subcaption[]{$v = 123$}
\end{subfigure}
\hfill
\vspace{0pt}
\begin{subfigure}{0.3\textwidth}

{
  \begin{tikzpicture}[xscale = 1.5, yscale = 1.5]

    \node[draw, circle, fill=black, inner sep=1.5pt, label=above:$a12$] (A1) at (0,1) {};
    \node[draw, circle, fill=black, inner sep=1.5pt, label=above:$a13$] (B1) at (1.5,1) {};
    \node[draw, circle, fill=black, inner sep=1.5pt, label=above:$a1n$] (D1) at (3,1) {};
    \node[] (E01) at (2,1) {};
    \node[] (E02) at (2.5,1) {};
  
    \draw[dotted] (E01) -- (E02);

    \node[draw, circle, fill=black, inner sep=1.5pt, label=below:$b12$] (A3) at (0,-1) {};
    \node[draw, circle, fill=black, inner sep=1.5pt, label=below:$b13$] (B3) at (1.5,-1) {};
    \node[draw, circle, fill=black, inner sep=1.5pt, label=below:$b1n$] (D3) at (3,-1) {};
    \node[] (E03) at (2,-1) {};
    \node[] (E04) at (2.5,-1) {};

    \draw[black, thick] (A3) -- (B1);
    \draw[black, thick] (A3) -- (D1);
 
  \draw[black, thick] (B3) -- (A1);
    \draw[black, thick] (B3) -- (D1);
  
    \draw[black, thick] (D3) -- (B1);
    \draw[black, thick] (D3) -- (A1);
    
  \draw[dotted] (E01) -- (E02); 
  \draw[dotted] (E03) -- (E04);

\end{tikzpicture}
}
\subcaption[]{$v = ab1$}
\end{subfigure}
\hfill

\caption{The links $\Lk(v, \cConf_3(\Theta_n))$ of vertices of $\cConf_3(\Theta_n)$.}
\label{fig:full_link}
\end{figure}
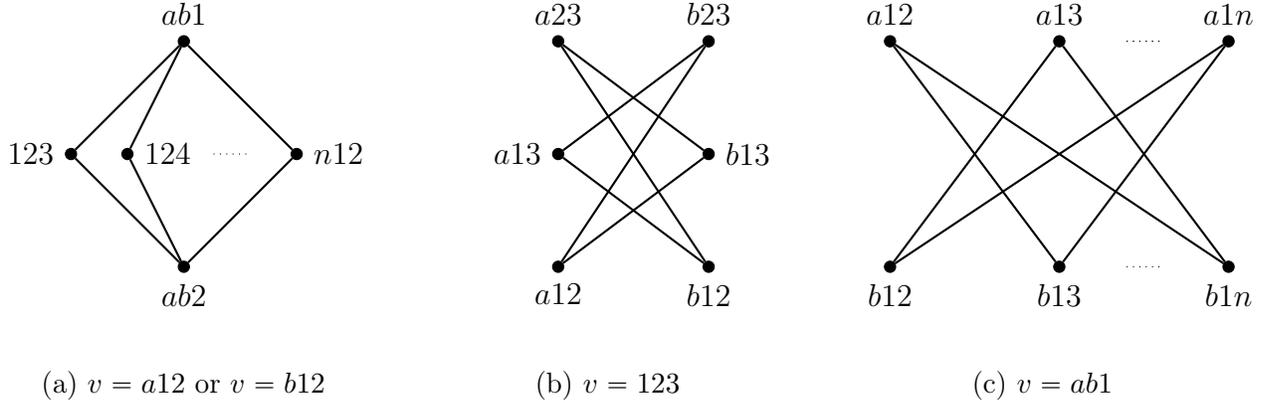

\vspace{-8pt}

\begin{remark}\label{eulerchar}
There is an alternate proof of Corollary~\ref{cor:3m} using tools from the theory of $3$-manifolds, as outlined in the next paragraph. However, Theorem~\ref{thm:main7} yields a stronger conclusion, as it detects the presence of an embedded non-planar graph in the visual boundary of the group up to quasi-isometry.

Let $G$ be an infinite non-cyclic, finitely generated group that does not split as a free product. Using tools from the theory of $3$-manifolds, if $G$ is $3$-manifold group, it must be the fundamental group of an aspherical compact oriented irreducible $3$-manifold with nonpositive Euler characteristic. 
Since the link of every vertex of $\cConf_3(\Theta_n)$ is connected and has no cut vertices or edges (see Figure~\ref{fig:full_link}), hence, by ~\cite[Theorem E]{wilton2024surface}, the group $B_3(\Theta_m)$ is one-ended and therefore does not split as a free product. Thus, $B_3(\Theta_m)$ is not a $3$-manifold group for $m>7$ since $\chi(B_3(\Theta_m))>0$ for $m>7$.

Suppose now that $G$ is also hyperbolic with $\chi(G) = 0$ and there is a $3$-manifold such that $\pi_1(M) \cong G$. Since $\chi(M) = 0$, we have $\chi(\partial M)=0$. Because $M$ is aspherical and irreducible, each boundary component must be an incompressible torus. Thus, we get an injective map from $\mathbb{Z}^2$ into $G$, which contradicts $G$ being hyperbolic. Since $B_3(\Theta_7)$ is one such group, it cannot be a $3$-manifold group.

\end{remark}

\bibliographystyle{amsalpha}
\bibliography{refs}
\end{document}